\documentclass[12pt,reqno]{amsart}
\advance \topmargin by -\headheight
\advance \topmargin by -\headsep
\evensidemargin \oddsidemargin
\usepackage{amsmath}
\usepackage{amsfonts}
\usepackage{stmaryrd}
\usepackage{amssymb}
\usepackage{amsthm}
\usepackage{csquotes}
\usepackage{color}
\usepackage{mathrsfs}
\usepackage{dsfont}
\usepackage{bm}
\usepackage{cite}
\usepackage{soul}
\usepackage{hyperref}

\numberwithin{equation}{section}
\renewcommand\vec{\bm}
\newcommand{\n}[1]{\|{#1}\|}

\newtheorem{theorem}{Theorem}[section]
\newtheorem{lemma}[theorem]{Lemma}
\newtheorem{Proposition}[theorem]{Proposition}

\newtheorem{Corollary}[theorem]{Corollary}
\newtheorem{Question}[theorem]{Question}

\theoremstyle{remark}
\newtheorem{remark}[theorem]{Remark}

\theoremstyle{definition}

\usepackage{mathtools}

\title{A Structure theorem for sets with doubling $4+\delta$}

\author{Yifan Jing}
\address{Department of Mathematics, The Ohio State University, Columbus OH, 43210 USA. }
\email{jing.245@osu.edu}

\author{Akshat Mudgal}
\address{Mathematics Institute, Zeeman Building, University of Warwick, Coventry CV4 7AL, United Kingdom}
\email{Akshat.Mudgal@warwick.ac.uk}


\renewcommand\vec{\bm}

\begin{document}

\def\ZZ{\mathbb{Z}}
\def\RR{\mathbb{R}}
\def\NN{\mathbb{N}}
\def\QQ{\mathbb{Q}}
\def\FF{\mathbb{F}}
\def\CC{\mathbb{C}}
\def\TT{\mathbb{T}}
\def\supp{\rm supp}

\def\C{\mathcal{C}}
\def\P{\mathcal{P}}
\def\L{\mathcal{L}}
\def\HH{\mathcal{H}}
\def\A{\mathcal{A}}

\def\eps{\varepsilon}

\def\M{\mathrm{Mat}}

\def\xxi{\vec{\xi}}
\def\eeta{\vec{\eta}}
\def\ggamma{\vec{\gamma}}
\def\Supp{\mathrm{Supp}}

\def\TT{\mathbb{T}}
\def\RR{\mathbb R}
\def\d{\,\mathrm d}
\def\ZZ{\mathbb{Z}}
\def\cS{\mathcal{S}}
\def\cY{\mathcal{Y}}
\def\cB{\mathcal{B}}

\newcommand{\Fr}[1]{\widehat{#1}}
\newcommand\1{\mathds{1}}

\def\cF{\mathcal{F}}

\begin{abstract}
    We prove a structural result for sets of integers with doubling at most $4 + \delta$, with $\delta>0$ sufficiently small. This generalises earlier work of Eberhard--Green--Manners which dealt with sets of integers with doubling strictly less than $4$, and makes progress towards a question of Green.
        \end{abstract}

\maketitle

\section{Introduction}

This paper concerns a structural theorem for sets of integers with small doubling. Here, given a finite, non-empty set $A$ of integers, we define the sumset $A+A$ of $A$ as
\[ A+A = \{a + a' : a,a' \in A\} \]
and the doubling $\sigma[A]$ of $A$ as $\sigma[A] = |A+A|/|A|$. A central theme in additive combinatorics concerns a classification of sets $A$ with $\sigma[A] \leq K$ for some parameter $K \geq 1$. A key example here is a higher dimensional progression. In particular, given $d \in \mathbb{N}$, we define a \emph{$d$--dimensional progression} $P$ to be a set of the form 
\begin{equation} \label{dimprog}
P = \{ a_0 + l_1 a_1 + \dots + l_d a_d \ : \  1 \leq l_i \leq L_i \ \ (1 \leq i \leq d)\}, 
\end{equation}
where $a_0, a_1, \dots, a_d \in \mathbb{Z}$ and $L_1, \dots, L_d \in \mathbb{N}$. Moreover, we say that $P$ is \emph{proper} if $|P| = L_1 L_2 \dots L_d$. One can see that proper $d$-dimensional progressions $P$ satisfy $\sigma[P] \leq 2^d$. It turns out that these are essentially the only examples; Freiman's theorem states that any set $A \subseteq \mathbb{Z}$ with $\sigma[A] \leq K$ must have a large subset contained in a proper $d$-dimensional progression $P$ such that $d \ll_K 1$ and $|P| \ll_K |A|$. Obtaining better quantitative dependence on $K$ in this result and variations thereof forms a key area of research in additive combinatorics, in part due to its many applications in combinatorics and number theory, see \cite{Go1998, GGMT2025, Ra2025, Sa2013, Mu2024} for further details and references.

In his list of problems, Green \cite[Problem 30]{Gr} inquired about a specific variation of this question. We record this question below.

\begin{Question} \label{grprob}
    Let $K \geq 2$, let $A \subset \mathbb{Z}$ be a finite set, let $\varepsilon >0$. If $|A-A| \leq K|A|$, then find a structured  subset $A' \subseteq A$ which `obviously' satisfies $|A'-A'| \leq (K + \eps)|A'|$.
\end{Question}

In order to give some indication of what kind of result to expect, we consider Freiman's $3k-3$ theorem, see \cite{LS1995, TV2006}. This states that any finite set $A \subseteq \mathbb{Z}$ which satisfies $|A-A| < 3|A| - 3$ must be contained in an arithmetic progression of size $|A-A| - |A| +1$. Let $\sigma_{-}[A] = |A-A|/|A|$. Thus, given $\delta \in (0,1)$ and $|A|$ sufficiently large in terms of $\delta,$ one see that whenever $\sigma_{-}(A) < 2 + \delta$, then there exists a $1$-dimensional progression $P$ with $|P| \geq |A|$ such that $A$ has density $1/(1 + \delta)$ on $P$. Moreover, any set $A' \subseteq P$ with $|A'| \geq |P|/(1 + \delta)$ trivially satisfies
\[ |A'-A'| \leq |P-P| < 2 |P| \leq (2 + 2\delta)|A'|.\]
Thus the upper bound on $\sigma_{-}[A']$ almost matches that of $\sigma_{-}[A]$.  On the other hand, Freiman's $3k-3$ theorem only holds for sets with doubling $< 3$. There is an old conjecture of Freiman \cite{Fr1999} concerning a generalisation of his $3k-3$ theorem for sets with doubling $< 10/3$, see also work of Jin \cite{Ji2007} who employed non-standard analysis to confirm this conjecture for large enough sets with doubling $< 3 + \varepsilon$, for some very small $\eps>0$.

The only known result of this type which concerns sets with doubling much larger than $3$ is a nice structural result due to Eberhard--Green--Manners\cite[Theorem 6.4]{EGM2014}. This states that for any $\varepsilon >0$ and any finite set $A \subseteq \mathbb{Z}$ satisfying $\sigma_{-}[A] \leq 4 - \varepsilon$, there exists a $1$-dimensional progression $P \subseteq \mathbb{Z}$ such that $|P| \gg_{\varepsilon} |A|$ and
\begin{equation} \label{opt4}
    |A \cap P|/|P| \geq 1/2 + c\varepsilon,
\end{equation}
for some absolute constant $c>0$. This structural result formed an important ingredient in the proof of their well-known result on the Erd\H{o}s sum-free set conjecture. Another nice aspect of this result is that given some small $\varepsilon>0$, any subset $A'$ of a $1$--dimensional progression $P$ with $|A'| \geq |P|(1/2 + c\varepsilon)$ obviously satisfies 
\[ |A'-A'| \leq |P-P| < 2|P| \leq (4 - c\varepsilon) |A'|,\]
and so, the conclusion recorded in \eqref{opt4} is optimal up to the constant $c>0$.

Thus, Eberhard--Green--Manners provide an affirmative answer to Question \ref{grprob} when $K < 4$. On the other hand, as soon as $K \geq 4$, one obtains at least two significantly different constructions that have doubling $K$. For instance, choosing $K = 4 + \delta$ for some  small $\delta \geq 0$, one can either let $A'$ be a subset of some $1$--dimensional progression $P$ with $|A'| \geq (1/2 - \delta)|P|$, which would give us 
\begin{equation} \label{1dimexample}
    |A'-A'| \leq |P-P| < 2|P| \leq (4 + \delta) |A'|,
\end{equation}
or alternatively, one could let $A'$ be a subset of a proper $2$--dimensional progression $Q$ with $|A'| \geq (1 - \delta/2)|Q|$, whence, 
\begin{equation} \label{2dimexample}
 |A'-A'| \leq |Q'-Q'| < 4 |Q'| \leq (4 +  \delta)|A'|. 
 \end{equation}

It is natural to ask whether these are the only two possible examples with doubling at most $4 + \delta$, for sufficiently small $\delta$. Our main result confirms this in the affirmative.

\begin{theorem} \label{main}
    Let $\varepsilon', \delta>0$ be real numbers with $\delta$ sufficiently small in terms of $\varepsilon'$. Let $A \subset \mathbb{Z}$ be a finite, non-empty set such that $|A+A| \leq (4 + \delta) |A|$. Then one of the following must hold. 
    \begin{enumerate}
        \item There exists a $1$--dimensional progression $P \subseteq \mathbb{Z}$ such that $|P| \gg_{\delta} |A|$ such that 
        \[ \frac{|A \cap P|}{|P|} \geq 1/2 - \varepsilon'.\]
        \item There exists a proper $2$--dimensional arithmetic progression $Q \subseteq \mathbb{Z}$ with $|Q| \gg_{\delta} |A|$ such that
        \[ \frac{|A \cap Q|}{|Q|} \geq 1 - \varepsilon'. \]
    \end{enumerate}
Moreover the same conclusion holds with $A+A$ replaced by $A-A$.
\end{theorem}

This provides an affirmative answer to Question \ref{grprob} whenever $K = 4 + \delta$, for sufficiently small $\delta>0$. Furthermore, noting \eqref{1dimexample} and \eqref{2dimexample}, we see that our result is optimal up to to the quantitative dependence between $\delta$ and $\varepsilon'$. 

While we have stated Theorem \ref{main} for finite sets of integers, this holds equally well for finite sets of real numbers. Indeed, every finite subset of reals is Freiman $20$-isomorphic to some finite set of integers, and since doubling and the additive structure of $d$-dimensional progressions remain invariant under such maps, this immediately gives Theorem \ref{main}  for finite sets of real numbers.

We briefly mention some applications. Thus, given some finite set $A \subset \mathbb{Z}$ and some $\Gamma \subseteq A \times A$, we define the restricted sumset 
\[ A+_{\Gamma} A = \{a_1 + a_2 : (a_1, a_2) \in \Gamma \}.\]
With this definition in hand, we state an ``almost-all" version of Theorem \ref{main}.

\begin{Corollary} \label{almostall}
  Let $1>\varepsilon' > \delta > \delta'$ be a decreasing sequence of positive real numbers with each term being sufficiently small in terms of the previous one, and let $A \subset \mathbb{Z}$ be a finite, non-empty set. Suppose there exists some $\Gamma \subseteq A \times A$ such that 
  \[ |\Gamma| \geq (1 -\delta')|A|^2 \ \ \text{and} \ \ |A +_{\Gamma} A| \leq (4 + \delta)|A|. \]
  Then one of the following must hold. 
    \begin{enumerate}
        \item There exists a $1$--dimensional progression $P \subseteq \mathbb{Z}$ such that $|P| \gg_{\delta} |A|$ such that 
        \[ \frac{|A \cap P|}{|P|} \geq 1/2 - \varepsilon'.\]
        \item There exists a proper $2$--dimensional arithmetic progression $Q \subseteq \mathbb{Z}$ with $|Q| \gg_{\delta} |A|$ such that
        \[ \frac{|A \cap Q|}{|Q|} \geq 1 - \varepsilon'. \]
    \end{enumerate}
\end{Corollary}

In order to give some context about this, note that the Balog--Szemer\'{e}di--Gowers theorem \cite[Theorem 4.1]{SSV2005} states that given parameters $C_1, C_2>0$, whenever $|\Gamma| \geq |A|^2/C_1$ and $|A+_{\Gamma} A| \leq C_2|A|$, one can find sets $A', B' \subseteq A$ such that 
\[ |A'|, |B'| \gg |A|/C_1^2 \ \ \text{and} \ \ |A'+B'| \ll C_1^5 C_2^3 |A|. \] 
In particular, the Balog--Szemer\'{e}di--Gowers theorem \cite{Go1998} finds large structured subsets of sets $A \subset \mathbb{Z}$ with large additive energy $E(A)$, where $E(A)$ counts the number of quadruples $(a_1, \dots, a_4) \in A^4$ such that $a_1 + a_2 = a_3 + a_4$, see \cite{RS2024} for the best quantitative results towards this. In \cite{Sh2019} Shao proved an almost-all version of the Balog--Szemer\'{e}di--Gowers theorem, where upon allowing $\Gamma$ to satisfy $|\Gamma| \geq (1 - o(1))|A|^2$, one can find $A', B' \subseteq A$ such that 
\[ |A'|, |B'| \geq (1 - o(1))|A| \ \text{and} \ |A'+A'| \leq |A +_{\Gamma} A| + o(|A|). \]
In fact, Corollary \ref{almostall} follows from combining Theorem \ref{main} and a straightforward consequence of \cite[Theorem 1.1]{Sh2019}, see \cite[Lemma 2.6]{JM2023}.

It is perhaps worth mentioning that our proof also naturally yields the following Bogolyubov--Ruzsa type lemma.

\begin{Corollary} \label{bogol1}
    Let $\delta>0$ be sufficiently small. Let $A \subseteq \mathbb{Z}$ be a finite, non-empty set such that $|A+A| \leq (4 + \delta) |A|$. Then $41A-40A$ contains a $d$-dimensional progression $P$ with $d \leq 2$ and $|P| \gg_{\delta} |A|$. 
\end{Corollary}

As in the case of Theorem \ref{main} the nice aspect of this result, compared to the usual Bogolyubov--Ruzsa lemma \cite{Sa2012}, is that in our specific setting we are able to get a sharp upper bound on the dimension $d$ of our progression at the cost of considering $41A - 40A$ instead of $2A-2A$. The proof of Corollary \ref{bogol1} is mentioned at the end of \S\ref{modelsection}.

It is natural to compare Theorem \ref{main} with the usual Freiman-type results. For example, for sets $A \subseteq \mathbb{Z}$ with $\sigma[A] \leq 4 + \delta$ for some small $\delta >0$, the Freiman--Bilu theorem \cite{Bi1999, GT2006} allows one to find a proper $d$-dimensional progression $P$ such that $|P| \gg |A|$ and $d \leq 2$ and $|A \cap P|/|P| \gg 1$. Thus at the cost of finding a potentially smaller progression, Theorem \ref{main} can precisely estimate the density of $A$ in $P$. An optimal understanding of this density seems to be the key theme of Question \ref{grprob}, and in fact, an affirmative answer to  Question \ref{grprob} for arbitrary $K>1$ does not seem to follow even conditionally on the polynomial Freiman--Ruzsa conjecture over the integers \cite[Conjecture 1.5]{Sa2013}.

It might also be tempting to compare Theorem \ref{main} to an inverse Brunn--Minkowski type result in $\mathbb{R}^2$ \cite{FJ2015, vHST2024} due to the two dimensional nature of our result. We remark that such results are much harder to prove in $\mathbb{Z}$, in part due to the fact that direct analogues of the continuous results often fail to hold in the discrete setting. For example, given integer $d \geq 2$ and some compact set $\mathcal{A} \subset \mathbb{R}^d$, the classical Brunn--Minkowski inequality implies that $\nu_d(\mathcal{A} - \mathcal{A}) \geq 2^d \nu_d(\mathcal{A})$, where $\nu_d$ is the Lebesgue measure in $\mathbb{R}^d$. On the other hand, there exist arbitrarily large finite sets $A \subseteq \mathbb{Z}^d$ with $\dim(A) = d$, where $\dim(A)$ denotes the dimension of the affine span of $A$ over $\mathbb{R}$,  which satisfy
\[ \sigma_{-}[A] \leq 2d -2 + \frac{1}{d-1}  .\]
This doubling constant is much smaller than $2^d$ as $d$ grows. In fact, just proving that this upper bound is sharp was a problem of Ruzsa that was only resolved very recently, see \cite{CL2025, Mu2022}. 

Another aspect where our result differs from the continuous setting is that it is indeed necessary to pass to a dense structured subset in the statement of Theorem \ref{main}, that is, it is not possible to have the entirety of $A$ contained in $P$ or $Q$ in Theorem \ref{main}. Indeed, consider the set $A_1 = \{0,10N,2^N\} + [N]$ with $N \geq 100$, where $[N]$ denotes the set $\{1,2,3, \dots, N\}$. Then one can see that
\[ |A_1 + A_1| = 6 (2N - 1) \leq 4 |A_1| - 6. \]
Moreover $A_1$ is not contained in any $2$-dimensional progression $Q$ such that $|Q| \ll |A_1|$. In comparison, the quantitative inverse Brunn--Minkowski result of Figalli--Jerison \cite{FJ2015} in $\mathbb{R}^2$  implies that any compact set $\mathcal{A}$ with $\nu_{2}(\mathcal{A})>0$ and $\nu_2(\mathcal{A} + \mathcal{A}) \leq (4 + \delta) \nu_2(\mathcal{A})$ must satisfy $\mu_2(\mathcal{C} \setminus \mathcal{A}) \ll \delta^{c}$ for some constant $c>0$, where $\mathcal{C}$ is the convex hull of $\mathcal{A}$. Thus, in the continuous setting, the entirety of $\mathcal{A}$ gets covered efficiently by a structured set.

Despite these contrasting features with the continuous setting, it is perhaps interesting that part of our proof of Theorem \ref{main} roughly involves understanding its counterpart over $\mathbb{Z}/q\mathbb{Z} \times [0,1] \times \mathbb{T}^d$, where $q, d$ are bounded, positive integers, and $\mathbb{T}$ denotes the torus $\mathbb{R}/\mathbb{Z}$. This, in turn, involves a combination of the Brunn-Minkowski inequality in $\mathbb{R}^2$ along with inverse Kneser type results due to Tao \cite{Ta2018} in $\mathbb{T}^d$. This transference of settings is done via the arithmetic regularity lemma, see \cite{GT2010, EGM2014, Eb2015} for further details about the latter.

\subsection*{Proof Ideas}

We will now provide a very rough proof sketch of our main result, brushing many technical details and combinatorial and analytic manoeuvres under the rug. We first consider the case when the set $A \subset \mathbb{Z}$ is dense in $[N]$. In this case, our setup follows the work of Eberhard--Green--Manners \cite{EGM2014} where one applies the abelian arithmetic regularity lemma along with various properties of Lipschitz functions over nilsequences and functions with small $\ell^2$ and $U^2$ norms to reduce the question to studying an analogous problem for subsets $\mathcal{A}$ of $\mathbb{Z}/q \mathbb{Z} \times [M] \times \mathbb{T}^d$, where $M$ is large but in a controlled manner, and $q, d \leq M$ are positive integers. Here one should think of $[M]$ as a discretisation of the usual interval $[0,1]$ present in the application of the abelian arithmetic regularity lemma. We now partition 
\[ \mathcal{A} = \bigcup_{\substack{1 \leq  a \leq q, \\ i \in [M]}} \mathcal{A}_{a,i} \]
as fibres over $\mathbb{T}^d$, where almost all fibres $\mathcal{A}_{a,i}$ satisfy 
\[ \mu(\mathcal{A}_{a,i}) = |A \cap I_{a,i}|/|I_{a,i}| +o(1) ,\]
with $I_{a,i}$ being the arithmetic progression
\[ I_{a,i} = \{n \in [N] : n \equiv a \ ({\rm mod} \ q) \ \text{and} \  n/N \in ((i-1)/M, i/M] \} \]
and $\mu$ being the Lebesgue measure over $\mathbb{T}^d$. One may assume that $\mu(A_{a,i}) \leq 1/2 + o(1)$ for all $a,i$ since otherwise we would be obtain the desired density on a $1$--dimensional progression. 

In the setting of Eberhard--Green--Manners, the question now reduces to studying a Brunn--Minkowski--Kneser type problem for the set
\begin{equation}  \label{introset}
\bigcup_{\substack{1 \leq a,a'\leq q, \\ i,i' \in [M]}} (\mathcal{A}_{a,i} + \mathcal{A}_{a',i'} ).
\end{equation}
At this point, Eberhard--Green--Manners split the fibres into two groups: the set of small fibres (those which satisfy $\mu(\mathcal{A}_{a,i}) < \varepsilon$) and the set of not-small fibres (those which satisfy $\mu(\mathcal{A}_{a,i} )\geq \varepsilon)$. Here $\varepsilon>0$ is some sufficiently small parameter. They are then able to safely ignore the small fibres and apply a combination of Brunn--Minkowski and Kneser type inequalities along with various clever combinatorial and analytic arguments on the not-small fibres to derive the desired conclusion of having doubling at least $4 - \varepsilon$.

This is where we are required to introduce further novelties, since our setting necessitates us to prove an \emph{inverse} Brunn--Minkowski--Kneser type result for the set described in \eqref{introset}. Thus, we introduce two new parameters of smallness: $\varepsilon'$ and $\delta$ such that $1> \varepsilon' > \delta > \varepsilon$, where each subsequent term is sufficiently small in comparison to the previous terms. With this in hand, we split the set of not-small fibres into medium-sized fibres (those which satisfy $\varepsilon < \mu(\mathcal{A}_{a,i}) \leq 10 \varepsilon'$) and large-sized fibres (those which satisfy $10 \varepsilon' \leq \mu(\mathcal{A}_{a,i})$). Using a Markov type inequality, one can quickly prove that there are $\gg qM$ many large-sized fibres. Amongst the set of large-sized fibres, we choose the biggest fibre, say $\mathcal{A}_{a_0,i_0}$ and consider the sums $\mathcal{A}_{a_0,i_0} + \mathcal{A}_{a,i}$, where $(a,i)$ varies amongst the large-sized fibres.  We apply an inverse Kneser result due to Tao \cite{Ta2018} along with a weighted Kneser-type inequality for these sumsets. This enables us to divide our proof into two cases. The first is the \emph{expansion} case, where we obtain a large doubling of the shape $4+ \delta$ for each of these sumsets. The second is the \emph{structured} case, where at least one of the large-sized fibres  $\mathcal{A}_{a,i}$ is very well-approximated by the inverse image of a $1$-dimensional interval in $\mathbb{T}$ under some homomorphism $\varphi: \mathbb{T}^d \to \mathbb{T}$ while also being the approximate support of some $1$-bounded Lipschitz function $F_{a,i}$.

In the expansion case, roughly speaking, our strategy allows us to obtain a doubling of the form $2+ \delta$ on a positive proportion of the sumset fibres $\mathcal{A}_{a,i} + \mathcal{A}_{a',i'}$, while obtaining a doubling of $2$ on the almost all the remaining sumset fibres. Here, it is quite important to obtain the stronger doubling on a positive proportion of the $qM$ many fibres since otherwise, we would not gain a strong enough doubling on the entire sumset $\mathcal{A} + \mathcal{A}$ that could dominate the $o(\delta)$ errors that we incur throughout this process. The above conclusion combines nicely with the Brunn--Minkowski setup of Eberhard--Green--Manners to derive a strong doubling estimate of the shape $4 + \delta - o(\delta)$ for $\mathcal{A} + \mathcal{A}$. We now use the arithmetic regularity lemma setup to pull this expansion back for our original sumset $A+A$.

The structured case requires some further new ideas. We begin by deriving a crucial piece of information: the function $F_{a,i}$ is almost $1$ on almost all of $\mathcal{A}_{a,i}$ while being almost $0$ everywhere else. This follows from a stability analysis of the aforementioned weighted Kneser-type inequality. Now, note that all homomorphisms $\varphi: \mathbb{T}^d \to \mathbb{T}$ satisfy $\varphi( \vec{x}) = \vec{a} \cdot \vec{x}$ for all $\vec{x} \in \mathbb{T}^d$, for some fixed non-zero vector $\vec{a} \in \mathbb{Z}^d$. Let $I \subseteq \mathbb{T}$ be the interval such that the large-sized fibre $ \mathcal{A}_{a,i}$ is well-approximated by $\varphi^{-1}(I)$. Our next aim is to show that $\vec{a}$ is well-controlled in size. This involves a combination of various technical convex geometric arguments along with some input from the geometry of numbers and the fact that $\mathcal{A}_{a,i}$ also roughly acts as the support of $F_{a,i}$. With this in hand, we can prove that the set $B = \{ n \in I_{a,i} : \varphi( \vec{\theta} n ) \in I \}$, where $\vec{\theta}$ is some \emph{$(\cF(M), N)$-irrational} vector in $\mathbb{T}^d$, is a large inhomogeneous Bohr set. Upon combining these ideas with properties of functions with small $\ell^2$ and $U^2$ norms as well as the fact that $F_{a,i}$ is almost $1$ on almost all of $\mathcal{A}_{a,i}$, we can prove that $A$ has density $1 - o(1)$ on $B$. We now want to prove that $B$ contains a proper $d$-dimensional progression $Q$ such that $|Q| \gg |B|$ and $d \leq 2$. Using the $(\cF(M),N)$-irrationality of $\vec{\theta}$ and the fact that $\vec{a}$ is controlled in size, we see that $\theta = \varphi(\vec{\theta})$ is $(\cF(M), N)$-irrational as well, wherein, equidistribution theory allows us to reduce our problem to the case when $B$ is a homogeneous Bohr set. We now use the theory of continued fractions as well as geometry of numbers to deduce the desired claim. Finally, since $A$ has a density $1 - o(1)$ on $B$, we immediately get that $A$ has a very high density on $Q$ as well. 

In fact, for the purposes of applying the arithmetic regularity lemma, it turns out that one needs to deal with popular sumsets instead of the entire sumset $A+A$ and study smoothened versions of characteristic functions. This forces us to incur various losses throughout the above proof and an important technical aspect of the above strategy is to control these losses by $o(\delta)$, since we only win a factor of $\delta$ over the Brunn--Minkowski--Kneser lower bound. 

Finally, in order to deduce the result for arbitrary sets $A \subset \mathbb{Z}$, we need to prove a Bogolyubov-type lemma for $2$-dimensional progressions. This requires multiple applications of a $1$-dimensional Bogolyubov-type result of Lev \cite{Lev1997} along with an analogous result of Eberhard--Green--Manners, both further amalgamated with a popularity type argument done along all the $1$-dimensional fibres of the $2$-dimensional progression. We combine this Bogolyubov-type lemma along with an application of a nice modelling lemma due to Green--Ruzsa \cite{GR2006} to conclude the proof.

\subsection*{Outline} In \S\ref{modelsection}, we reduce the proof of the general case of Theorem \ref{main} to the setting when $A$ is dense in $[N]$. We use \S\ref{reglemmasteup} to record the setup for applying the arithmetic regularity lemma and using to transfer our setting to  $\mathbb{Z}/q\mathbb{Z} \times [M] \times \mathbb{T}^d$. \S\ref{invKnesersection} is used for splitting the fibres into different sized classes and for applications of the inverse Kneser type result of Tao as well as weighted Kneser type inequalities. This is where we are able to divide our proof into two major cases, the Expansion case and the $2$-dimensional structured case. We resolve the expansion case in \S\ref{expansionsection}. We begin our analysis of the $2$-dimensional structured case in \S\ref{structuresection} where we employ a variety of convex geometric arguments to control the operator norm of the homomorphism $\varphi: \TT^d \to \TT$. In \S\ref{bohrsection}, we use this information along with various other arguments to prove that $A$ has density $1- o(1)$ on some inhomogeneous Bohr set. We employ \S\ref{continuedfracsection} to pass from the inhomogeneous Bohr set to a proper $2$-dimensional progression. We also supply some preliminary results concerning the abelian arithmetic regularity lemma in Appendix \ref{arithmeticregularityappendix}.

\subsection*{Notation} 
We use Vinogradov notation, that is, we write $X \ll_z Y$,  or equivalently $X= O_z(Y)$, to mean that $|X| \leq C_z Y$, where $C_z>0$ is some constant depending on the parameter $z$. For any $\theta \in \mathbb{R}$, we denote $e(\theta) = e^{2 \pi i\theta}$. For any finite subset $X$ of some abelian group and any $k \in \mathbb{N}$, we write $X^k = \{(x_1, \dots, x_k) : x_1, \dots, x_k \in X\}$. Moreover, we will use $\vec{v}$ to denote the vector $(v_1, \dots, v_k) \in X^k$. For any $n \in \mathbb{Z}$ and any $\vec{v}, \vec{u} \in X^k$, we write $n \vec{v} = (nv_1, \dots, nv_k)$ and $\vec{v} \cdot \vec{u} = v_1 u_1 + \dots + v_k u_k$.   Given a positive real number $X$, we use $[X]$ to denote the set $\{1,2,\dots, \lfloor X \rfloor\}$.

\subsection*{Acknowledgements}
YJ and AM were partly supported by Ben Green’s Simons Investigator Grant, ID:376201. AM is supported by a Leverhulme early career fellowship \texttt{ECF-2025-148}.



\section{Dense model lemma} \label{modelsection}

Using various results from additive combinatorics, we can reduce the proof of Theorem \ref{main} to the setting when $A$ is a dense subset of $[N]$.

\begin{theorem} \label{m1}
    Let $\alpha > \varepsilon' > \delta > \varepsilon$ be a decreasing sequence of positive real numbers with each term being sufficiently small in terms of the previous one. Then every $A \subseteq [N]$ with $|A| = \alpha N$ satisfies at least one of the following. 
    \begin{enumerate}
        \item \label{it1} We have $|A+A| > (4 + \delta)|A|$.
        \item \label{it2} There exists an arithmetic progression $P \subseteq [N]$ with $|P| \gg_{\varepsilon} N$ such that 
        \[ \frac{|A \cap P|}{|P|} \geq 1/2 - O(\varepsilon').\]
        \item \label{it3} There exists a proper $2$-dimensional arithmetic progression $Q \subseteq [N]$ with $|Q| \gg_{\varepsilon} N$ such that
        \[ \frac{|A \cap Q|}{|Q|} \geq 1 - O(\varepsilon'). \]
        Moreover, the same conclusion holds with $A+A$ replaced by $A-A$.
    \end{enumerate}
\end{theorem}

We will utilise this section to show how one may proceed with this reduction in the case when we analyse sumsets $A+A$ in Theorems \ref{main} and \ref{m1}. The reduction in the case of difference sets $A-A$ follows in a very similar and, in fact, slightly simpler fashion.

Thus, let $l \in \mathbb{N}$ and let $X \subset \mathbb{Z}$ be a set such that $|X| \geq 2$, and $X \subseteq \{0,1,2,\dots, l\}$ with $0, l \in X$. Moreover, suppose that there exists no $n \in \mathbb{N}$ with $n \geq 2$ such that $n | x$ for all $x \in X$. Given $h \in \mathbb{N}$, we define
\[ hX = \{ x_1 + \dots + x_h : x_1, \dots, x_h \in X\}. \]
In this setting, we have the following result due to Lev \cite[Lemma 1]{Lev1997}. 

\begin{lemma} \label{lev}
    Let $k,r \in \mathbb{N}$ satisfy 
    \[ k \leq \frac{l-1}{|X|-2} \leq k +1  \ \ \text{and} \ \ r = (k+1)(|X| - 2) - (l-2).  \]
    Then 
    \[   2k X \supseteq [ kl - kr, kl + kr] \cap \mathbb{Z}, \ \ \text{and} \ \  (2k+1)X \supseteq  [ kl - kr, (k+1)l + kr] \cap \mathbb{Z}. \]
\end{lemma}

Eberhard--Green--Manners \cite[Lemma 6.3]{EGM2014} used the above lemma to prove the following Bogolyubov type result. 

\begin{lemma} \label{egmcover}
    Let $P$ be a $1$-dimensional progression in $\mathbb{Z}$ with $|P| \geq 12$, let $X \subseteq P$ be a set such that $|X|/|P| >1/2$. Then $5X - 4X$ contains $P$.
\end{lemma}

The density condition in this lemma is sharp; indeed, if $P = [2l]$ and $X = 2 \cdot [l]$, then $mX - nX$ will never contain $1$. On the other hand, if we further assume that $P$ is  the smallest $1$-dimensional progression containing $X$, we can go beyond this density barrier.

\begin{lemma} \label{1dimpullback}
   Let $X \subseteq \mathbb{Z}$ be a finite set of integers with $|X| \geq 100$, and let $P$ be the smallest $1$-dimensional progression that contains $X$ with  $|X| /|P|>2/5$. Then $ 9X - 8X$ contains $P$.
\end{lemma}

\begin{proof}
Suppose $P= u + v \cdot \{0,1,2,\dots, l\}$ with $u \in \mathbb{Z}$ and $v \in \mathbb{N}$. Then 
\[ P \subseteq X  + 8X - 8X \ \ \Leftrightarrow \ \ \frac{1}{v} \cdot (P - u) \subseteq \frac{1}{v} \cdot (X-u) + \left(8 \left(\frac{1}{v} \cdot( X-u) \right) - 8 \left( \frac{1}{v} \cdot(X-u) \right) \right) . \]
Here, and in the sequel, we define $\lambda \cdot X = \{ \lambda x : x \in X\}$ for any $\lambda \in \mathbb{R}$ and any subset $X$ of some vector space over $\mathbb{R}$.     Thus, we may assume that $P = \{0,1,\dots, l\}$ for some $l \in \mathbb{N}$ satisfying $l \geq |X| \geq 100$, and that $0, l \in X$. If $|X| > |P|/2$, we can directly use Lemma \ref{egmcover} to get that $P \subseteq 5X - 4X \subseteq 9X- 8X$. Thus, we can assume that $|X|/|P| \leq 1/2$. Since 
    \[ 2/5  < |X|/|P| \leq 1/2,\]
    we can then apply Lemma \ref{lev} with $k = 2$ and $r \in \mathbb{Z}$ satisfying
    \[
   \frac{l}{5} - 4 < r = 3|X| - l - 4 \leq \frac{l}{2}-4. 
\]
We thus deduce that $4X$ contains an interval $I$ of length $4r+1$. 
 This implies that $8X - 8X$ contains $2I-2I$ which itself contains a symmetric interval  $[-8r,8r]\cap \mathbb{Z}$. Note that $8r> 8l/5 - 4 > l$, and so $2I-2I$ contains $\{0,1,\dots, l\}$. We conclude our proof by noting that $0 \in X$, and so, we have $X + 8X - 8X \supseteq 8X-8X \supseteq P$.    
    \end{proof}

We now present our second Bogolyubov type result.

\begin{lemma} \label{2dimpullback}
    Let $c \in (0,1/10)$ be a real number,  let $Q$ be a proper $2$--dimensional progression, and let $X \subseteq Q$ be a finite, non-empty set such that $|X| \geq 100$. Then if $|X| \geq |Q|(1 - c)$, then $Q \subseteq 41X - 40X$. 
\end{lemma}

\begin{proof}
We may write $Q = u + v_1 \cdot \{0,1,\dots,l_1\} + v_2 \cdot \{0,1,\dots, l_2\}$ for some $l_1, l_2 \in \mathbb{N}$. As in the previous lemma, we can perform an affine translation to assume that $u=0$. For any $0 \leq i \leq l_1$, we define
    \[ Q_i =  i v_1 + v_2 \cdot \{0,1,\dots, l_2\} \ \ \text{and} \ \ X_i = X \cap Q_i. \]
    Writing 
    \[ I_1 = \{0 \leq i \leq l_2 : |X_i|/|Q_i| > 1/2 \}, \]
    a standard popularity-type argument gives us
    \[ \frac{(l_2 +1)( l_1 + 1 - |I_1|)}{2}  + (l_2 + 1)|I_1| \geq \sum_{i=0}^{l_1} |X_i| = |X| \geq |Q|(1-c). \]
    Since $Q$ is proper, we get that $|Q| = (l_1+1)(l_2 + 1)$, and so, the preceding expression gives us
    \begin{equation} \label{marsfortherich}
    |I_1| \geq (l_1 + 1)(1 - 2 c). 
    \end{equation}
    For each $i \in I_1$, we may apply Lemma \ref{egmcover} to deduce that $Q_i \subseteq 5X_i - 4X_i$.

    Now let $X' = 5X - 4X$, and for every $0 \leq j \leq l_2$, let
    \[ Q_j' =  v_1 \cdot \{0,1,\dots, l_1\} + jv_2. \]
    Noting \eqref{marsfortherich}, we see that for any $0 \leq j \leq l_2$, we have that 
    \[ |X' \cap Q_j'|/|Q_j'| \geq |I_1|/(l_1+1) \geq  1 - 2 c. \]
    Thus, we may now apply Lemma \ref{egmcover} to deduce that for any $0 \leq j \leq l_2$, one has 
     \[ Q_j' \subseteq 5X' - 4X' = 41X - 40X  .\]
    This implies that $Q = \bigcup_{j=0}^{l_2} Q_j' \subseteq 41X - 40X$.
\end{proof}

In order to deduce Theorem \ref{main} from Theorem \ref{m1}, we will need one further ingredient. In order to state this, we require the following auxiliary definition. Given $k \in \mathbb{N}$, we say that two finite, non-empty  sets $A, B \subseteq \mathbb{Z}$ are \emph{Freiman $k$--isomorphic} if there exists some bijection $\varphi : A \to B$ such that for all $a_1, \dots, a_{2k} \in A$, we have
\[ a_1 + \dots + a_k = a_{k+1} + \dots + a_{2k} \ \ \text{if and only if} \ \ \varphi(a_1) + \dots + \varphi(a_k) = \varphi(a_{k+1}) + \dots + \varphi(a_{2k}) \]
It is easy to see that if $A$ and $B$ are Freiman $k$--isomorphic for some $k \in \mathbb{N}$, then in fact they are Freiman $l$--isomorphic for any integer $l \in [k]$. Moreover, one has that if $A$ and $B$ are Freiman $k$--isomorphic and $l,m$ are integers such that $k > l+m$, then the set $lA - mA$ is Freiman $r$--isomorphic to $lB - mB$ for any integer $1\leq r \leq k/(l+m)$. We refer the reader to \cite[\S5.3]{TV2006} for more details concerning Freiman isomorphisms.

With this in hand, we now record the following rectification type result due to Green--Ruzsa \cite{GR2006}. 

\begin{lemma} \label{greenruzsa}
    Let $k  \in \mathbb{N}$, let $K \geq 1$ and let $A$ be a finite, non-empty set of integers such that $|A+A| \leq K|A|$. Then $A$ is Freiman $k$-isomorphic to some subset of the interval $[(16kK)^{12K^2}|A|]$.
\end{lemma}

We will now combine these results to present the proof of Theorem \ref{main}.

\begin{proof}[Proof of Theorem \ref{main}]
Let $\eps''>0$ be a sufficiently small number to be chosen later, let $\delta>0$ be sufficiently small in terms of $\eps''$.
Let $A \subseteq \mathbb{Z}$ be a finite, non-empty set satisfying $|A+A| \leq (4+ \delta)|A|$. Applying Lemma  \ref{greenruzsa}, we find that $A$ is Freiman $2000$--isomorphic to a set $A' \subseteq \mathbb{Z}/N\mathbb{Z}$ for some $N \ll |A| = |A'|$. Let this Freiman $2000$--isomorphism be defined by the map $\varphi : A \to A'$. Note that Freiman isomorphisms preserve the size of sumsets and so, we have that $|A'+A'| \leq (4+ \delta)|A'|$. 

Now fix $\varepsilon>0$  which is sufficiently small in terms of $\delta>0$. Applying Theorem \ref{m1}, we find that either there exists some $1$--dimensional progression $P' \subseteq [N]$ such that 
\[ |P'| \gg_{\varepsilon} N \gg_{\varepsilon} |A'| \ \ \text{and} \ \ |A'\cap P'|/|P'| \geq 1/2 - O(\varepsilon'')\]
or there exists a proper $2$--dimensional progression $Q' \subseteq [N]$ with 
\[ |Q'| \gg_{\varepsilon} N  \gg_{\eps} |A'| \ \ \text{such that}  \ \ |A' \cap Q'|/|Q'| \geq 1 - O(\varepsilon''). \]

Suppose we are in the first situation. Note that by passing to a sub--progression, we may further assume that $P'$ is the smallest sub--progression containing $A'$. We may now apply Lemma \ref{1dimpullback} to deduce that $P' \subseteq 9A'-8A'$. On the other hand, since $A$ and $A'$ are Freiman $2000$--isomorphic, we get that $9A' - 8A'$ is Freiman $100$--isomorphic to $9A - 8A$. Let this Freiman $100$--isomorphism be given by the map $\varphi_1 : 9A'- 8A' \to 9A - 8A$. Moreover, since Freiman $4$--isomorphisms preserve proper generalised arithmetic progressions, see \cite[Proposition 5.24]{TV2006}, we deduce that the set $P = \varphi_1(P')$ is a $1$--dimensional progression such that 
\[ |P| \gg_{\varepsilon} |A| \ \ \text{and} \ \ |A\cap P|/|P| \geq 1/2 - O(\varepsilon''). \]

Similarly, if we are in the second situation, then we can apply Lemma \ref{2dimpullback} to deduce that $Q' \subseteq 41A'-40A'$. As before, by noting that $41A'-40A'$ is Freiman $10$--isomorphic to $41 A - 40 A$, we may pull back this $2$-dimensional progression under this Freiman $10$--isomorphism to obtain a proper $2$--dimensional progression $Q$ such that
\[ |Q| \gg_{\varepsilon} |A| \ \ \text{and} \ \ |A \cap Q|/|Q| \geq 1 - O(\varepsilon''). \]
Letting $\eps'' = C_1 \eps'$ for some appropriate constant $C_1>0$  then finishes the proof of Theorem \ref{main}.
\end{proof}

As mentioned in the introduction, we see that the above proof actually yields a Bogolyubov--Ruzsa type lemma. Indeed, under the hypothesis of Theorem \ref{main}, we see that either $9A - 8A$ contains the $1$-dimensional progression $\varphi_1(P)$ with $|\varphi_1(P)| \gg_{\varepsilon}|A|$, or $41A - 40A$ contains a $2$-dimensional progression with size $\gg_{\varepsilon} |A|$, where $\varepsilon$ is some fixed real number which is sufficiently small in terms of $\delta$.

Our main goal now is to prove Theorem \ref{m1}. In order to maintain an appropriate analogy with Question \ref{grprob} as well as the work of Eberhard--Green--Manners \cite{EGM2014}, we will present our proof of Theorem \ref{m1} in the case when we analyse the difference set $A-A$. The proof in the case of the sumset $A+A$ also follows in a very similar fashion.


\section{Preliminary manoeuvres with the Regularity lemma} \label{reglemmasteup}

In this section, we closely follow the arguments of \cite[\S4]{EGM2014}. We begin by providing some further notation, and so, given a function $f: [N] \to \mathbb{C}$, we define
\[ \n{f}_{{\ell}^2} = (N^{-1}\sum_{n \in [N]} |f(n)|^2 )^{1/2}.\]
Writing $E(N)$ to count all quadruples $(a,b,c,d) \in [N]^4$ such that $a+d = b+c$, we also define
\[ \n{f}_{U^2}^4 = E(N)^{-1} \sum_{\substack{a,b,c,d \in [N], \\
a+ d = b+c}} f(a) \overline{f(b) f(c)} f(d). \]
The reader can verify that this definition is equivalent to  \cite[Definition A.7]{EGM2014}. In fact, one can check that for any prime $200N \leq p < 400N$, if one views $f$ as a function from $\mathbb{Z}/p \mathbb{Z}$ to $\mathbb{C}$ with support in $[N] \ ({\rm mod} \ p)$, then 
\begin{equation} \label{equivdefuniform}
\n{f}_{U^2}^4 \ll \sum_{r=1}^4 |\hat{f}(r)|^4 <  \n{f}_{U^2}^4  
\end{equation}
where $\hat{f}(r) = p^{-1} \sum_{x \in \mathbb{Z}/p \mathbb{Z}} f(x) e(rx/p)$ denotes the Fourier transform of $f$ at frequency $r$.

Given a vector $\vec{x} \in \mathbb{R}^d$ and $p \geq 1$, we define
\[ \n{\vec{x}}_p = (|x_1|^p + \dots + |x_d|^p )^{1/p} \ \ \text{and} \ \ \n{\vec{x}}_{\infty} = \max_{1 \leq i \leq d} |x_i|. \]
An element $\vec{\theta}\in \TT^d$ is \emph{$(A,N)$-irrational} if for every nonzero $\vec{m} = (m_1, \dots, m_d) \in \ZZ^d$ with $\n{\vec{m}}_1 \le A$ one has
\[
\|\mathbf m \cdot \vec{\theta}\|_{\TT}\ge \frac{A}{N},
\]
where $\|x\|_{\TT}$ denotes distance from $x\in \TT$ to the nearest integer. 

Now, let $\vec{x} = (a, y_0, \vec{y})$ and $\vec{x}' = (a', y_0', \vec{y}')$ be elements in the space $\mathbb{Z}/q\mathbb{Z} \times [0,1] \times \mathbb{T}^d$. We define a metric $d$ on this space by setting
\[ d(\vec{x}, \vec{x}') = \max\{ |a-a'|_q, |y_0 - y_0'|, \n{ \vec{y} - \vec{y'}} \}  ,\]
where 
\[ |a-a'|_q = \inf_{\substack{ n_1,n_2 \in \mathbb{Z},\\ n_1 \equiv a \ ({\rm mod} \ q),\\ n_2 \equiv a' \ ({\rm mod} \ q)  }} |n_1 - n_2| \ \ \text{and} \ \ \n{\vec{y} - \vec{y'}} = \inf_{\vec{m}, \vec{m}' \in \mathbb{Z}^2} \n{\vec{y} + \vec{m} - \vec{y}' - \vec{m}'}_2 .\] 
Moreover, given $M >0$, we denote a function $F: \mathbb{Z}/q\mathbb{Z} \times [0,1] \times \mathbb{T}^d \to \mathbb{R}$ to be  \emph{$M$-Lipschitz} if for all $\vec{x}, \vec{x}' \in \mathbb{T}^d$, one has
\[ |F(\vec{x}) - F(\vec{x}')| \leq M d(\vec{x}, \vec{x}'). \]
One can similarly define an $M$-Lipschitz function on the space $[0,1]\times \TT^d$ or $\TT^d$.

\subsection{A hierarchy of parameters} \label{hp1} 

Recall the hypothesis of Theorem \ref{m1}. Thus, we have three main parameters: $\varepsilon' > \delta > \varepsilon >0$, each of these being sufficiently small in terms of the previous number. Moreover $\varepsilon'$ itself is much smaller than $\alpha$, where $\alpha = |A|/N$, which is some fixed constant in $(0,1]$. 

There will also be a rapidly increasing growth function $\cF : (0, \infty) \to (0, \infty)$ which will grow sufficiently fast in terms of $\varepsilon', \delta,\varepsilon$. Thus, for instance, we can always assume that $\cF(1)$ is sufficiently large in terms of  $\varepsilon', \delta,\varepsilon$.

Throughout this paper, we will define a variety of auxiliary parameters which will depend on $\varepsilon', \delta,\varepsilon$ in reasonable ways. In particular, we will choose
\begin{equation} \label{cot}
\eta = \delta^{1000} \ \ \text{and} \ \  c= 1/2000 \ \ \text{and} \ \ \lambda = \varepsilon^2. 
\end{equation}
Our choice of these parameters implies that $\eta^c = \delta^{1/2}$, which itself will be another tertiary parameter that will come into play in later sections.

\subsection{Arithmetic regularity lemma}

We will need the following straightforward consequence of the abelian arithmetic regularity lemma, see \cite[(4.4)]{EGM2014}.

\begin{lemma} \label{arithlem}
Given parameters as described in \S\ref{hp1} and some set $A \subseteq [N]$ with $|A| = \alpha N$, there exists an integer 
\[ \varepsilon^{-10} \leq M \ll_{\varepsilon, \cF} 1,\]
some integers $1 \leq q,d \leq M$, some $M$-Lipschitz function 
\[ F: \mathbb{Z}/q \mathbb{Z} \times [0,1] \times  \TT^d \to [0,1],\]
some $(\cF(M), N)$-irrational $\vec{\theta} \in \TT^d$ such that 
\begin{equation} \label{arl1}
    \1_A = f_{\rm struct} + f_{\rm sml} + f_{\rm unf},
\end{equation} 
where $f_{\rm struct} : [N] \to [0,1]$ and $f_{\rm unf} : [N] \to [-1,1]$ and $f_{\rm sml}:[N] \to [-2,2]$ are functions satisfying
\begin{equation}  \label{arl2}
f_{\rm struct} = \sum_{a \ ({\rm mod} \ q) } \sum_{i =1}^M  \1_{n \in I_{a,i}} F_{a,i}( n \vec{\theta}) , \ \ \text{and} \  \ \n{f_{\rm sml}}_{{\ell}^2} \leq 2\varepsilon^{10}, \  \ \text{and} \  \  \n{f_{\rm unf}}_{U^2} \leq 1/\cF(M), 
\end{equation}
and 
\[ I_{a,i} = \{ n \in [N] : n/N \in ( (i-1)/M, i/M] \ \text{and} \ n \equiv a \ ( {\rm mod} \ q) \} \]
and $F_{a,i} (n \vec{\theta}) = F(a,i/M, n \vec{\theta})$ for every $a \in \mathbb{Z}/q\mathbb{Z}$ and $i \in [M]$ and $n \in [N]$.
\end{lemma}


Define $\alpha(a,i) = |A \cap I_{a,i}|/|I_{a,i}|$ for every $a \in \mathbb{Z}/q\mathbb{Z}$ and $i \in [M]$. The following is a nice consequence of the ${\ell}^2$ control in \eqref{arl2}, see  \cite[Lemma 4.3]{EGM2014}. 

\begin{lemma} \label{exceptionlemma}
There exists some $E \subseteq \mathbb{Z}/q\mathbb{Z}\times [M]$ with $|E| \leq \varepsilon^4 qM$ such that whenever $(a,i) \notin E$, then
\begin{equation} \label{l2sml}
\sum_{n \in I_{a,i}} |f_{\rm sml}(n)| \leq \varepsilon^5 |I_{a,i}| .
\end{equation}
\end{lemma}

This, along with various properties about functions with a small $U^2$ norm and functions that are $M$--Lipschitz and well-behaved on $\TT^d$ allow us to deduce the following, see also \cite[Lemma 4.4]{EGM2014}

\begin{lemma} \label{t1}
    For all $(a,i) \notin E$, we have
\[ \Big| \int_{\TT^d} F_{a,i}(\vec{\gamma}) d\vec{\gamma} - \alpha(a,i) \Big| \leq \varepsilon + qM/\cF(M).  \] 
\end{lemma}

\begin{proof}[Proof sketch]
Lemma~\ref{gowers-progressions} implies that
\[ |\sum_{n \in I_{a,i}} f_{\rm unf}(n) | < N / \cF(M). \]
Thus, the fact that
\[ \sum_{n \in I_{a,i}} F_{a,i}(n \vec{\theta}) = |A \cap I_{a,i}| - \sum_{n \in I_{a,i}} f_{\rm sml}(n) - \sum_{n \in I_{a,i}} f_{\rm unf}(n) \]
combines with the preceding inequality and \eqref{l2sml} to give us that 
\[ \sum_{n \in I_{a,i}} F_{a,i}(n \vec{\theta}) = \alpha(a,i)|I_{a,i}| + O(\varepsilon^5 |I_{a,i}|  + N/\cF(M) )\]
Moreover, Lemma~\ref{distribution-integral-a} implies that
\[ \Big||I_{a,i}|^{-1}\sum_{n \in I_{a,i}  } F_{a,i}(n \vec{\theta}) - \int F_{a,i}(\vec{\gamma}) d\vec{\gamma} \Big| < \varepsilon. \qedhere \]
\end{proof}


\subsection{Transference to $\mathbb{T}^d$}

We will now start our first transference step, where we move our problem from the setting of integers to $\TT^d$. We will then resolve a variation of our problem over $\TT^d$, and then pull either the expansion phenomenon or the structure back to the integer setting. 

 Recall that 
\[ A - A \subseteq [-N,N] = \bigcup_{\substack{1 \leq j \leq q, \\ l \in [-M,M] }} I_{j,l} . \]
Let $(a,i), (a',i') \notin E$ and let $d \in I_{a-a',i-i'}$. Noting \eqref{arl1} and \eqref{arl2}, one has
\begin{align*}
 \sum_{n \in I_{a',i'}}\1_{A_{a',i'}}(n ) \1_{A_{a,i}}(n+d)  &  = \sum_{n \in I_{a',i'}}   f_{a',i'}(n) f_{a,i} (n+d)  + \sum_{n \in I_{a',i'}} f_{\rm sml}(n) f_{a,i}(n+d) \\
 & +   \sum_{n \in I_{a',i'}} f_{\rm unf}(n) f_{a,i}(n+d) +  \mathcal{E},
\end{align*}
where $f_{a,i}(n) = F_{a,i}(n \vec{\theta})\1_{n \in I_{a,i}}$ and $\mathcal{E}$ only involves terms comprising $f_{\rm sml}$ and $f_{\rm unf}$. Since $(a,i), (a',i') \notin E$, we see that for any $g : [N] \to [-10,10]$, one has
\begin{equation} \label{l2err2}
\Big|\sum_{n \in I_{a',i'}} f_{\rm sml}(n) g(n+d) \Big|\leq 10 \sum_{n \in I_{a',i'}} |f_{\rm sml}(n)| < 10 \varepsilon^5 |I_{a,i}|,
\end{equation}
where the final inequality follows from \eqref{l2sml}. Moreover, for any such $g$, we also have
\[ \sum_{d \in [-N,N]} \Big|\sum_{n \in I_{a',i'}} f_{\rm unf}(n) g(n+d) \Big|^2  \ll N^3 \n{f_{\rm unf}}_{U^2}^2 \n{ g}_{U^2}^2 \ll N^3 \cF(M)^{-1},\]
wherein, we have used the fact that $\n{g}_{U^2} \leq \n{g}_{\infty} \ll 1$. Thus, for all $d \in [-N,N]$ outside of an exceptional set of size $\ll N/\cF(M)^{1/3}$, we have
\[ \Big| \sum_{n \in I_{a',i'} } f_{\rm unf}(n) g(n+d) \Big| \ll  N/ \cF(M)^{1/3}.  \]
The terms involved in $\mathcal{E}$ can be handled similarly.

Thus, for all $d \in [-N,N]$ outside of an exceptional set of size $O( N/\cF(M)^{1/3})$, one has
\begin{align} \label{newrev1}
    \sum_{n \in I_{a',i'}}\1_{A_{a',i'}}(n ) \1_{A_{a,i}}(n+d) &  =  \sum_{n \in I_{a',i'} \cap (I_{a,i}-d )}   f_{a',i'}(n) f_{a,i} (n+d) \nonumber \\
    & \qquad \qquad \qquad + O(\varepsilon^5 (|I_{a,i}| + |I_{a',i'}|) + N/\cF(M)^{1/3}) .
\end{align} 
Now, except $O(\varepsilon N/(qM))$ many $d \in I_{a-a,i-i'}$, we have that $I_{a,i} \cap (I_{a',i'} + d)$ is an arithmetic progression of size $\geq \varepsilon N/qM \gg \varepsilon |I_{a,i}|$.  We let $G'$ be the set of all these $d$. We observe that $G'$ is in fact an arithmetic progression of size $(N/qM ( 1- O(\eps) )$. Moreover, for any $d \in G'$, we can use Lemma~\ref{distribution-integral-a} along with the $(\cF(M),N)-$irrationality of $\vec{\theta}$ to deduce that
\begin{align} \label{fourvisions}
\sum_{n \in I_{a',i'} \cap (I_{a,i}-d )}   f_{a',i'}(n) f_{a,i} (n+d) & =  \sum_{n \in I_{a',i'} \cap (I_{a,i}-d )}   F_{a',i'}(n \vec{\theta}) F_{a,i} (n\vec{\theta} +d\vec{\theta})  \nonumber \\
& = | I_{a,i} \cap (I_{a',i'} + d)| \bigg( \int_{\TT^d} F_{a',i'}(\gamma) F_{a,i}(\gamma + d \vec{\theta})  + O(\varepsilon^5) \bigg) \nonumber \\
& = | I_{a,i} \cap (I_{a',i'} + d)| \bigg( F_{a,i} \circ F_{a',i'}(d \vec{\theta})  + O(\varepsilon^5) \bigg).
\end{align}

We will now apply Lemma \ref{proportionlem} and the fact that $F_{a,i} \circ F_{a',i'}$ is $M$--Lipschitz to deduce that the number of $d \in G'$ for which $F_{a,i} \circ F_{a',i'}(d \vec{\theta}) \gg \varepsilon^2 \eta^2$, where $\eta = \delta^{1000}$ (see  \eqref{cot}), is at least
\begin{equation*}
    |I_{a-a,i-i'}| \big( \mu( \{ \vec{x} \in \TT^d : F_{a,i} \circ F_{a',i'}(\vec{x} ) \gg \varepsilon^2 \eta^2 \})  - O(\varepsilon^2 \eta^2) \big). 
\end{equation*}
Since $\eta = \delta^{1000}$, we see that  $\varepsilon^5$ can be made to be much smaller than $\varepsilon^2 \eta^2$ by choosing $\varepsilon$ to be sufficiently small in terms of $\delta$ (in fact $\eps < \delta^{1000}$ suffices). Combining this with \eqref{fourvisions}, we see that  any $d \in G'$ satisfying $F_{a,i} \circ F_{a',i'}(d \vec{\theta}) \gg \varepsilon^2 \eta^2$ must also satisfy 
\[\sum_{n \in I_{a',i'} \cap (I_{a,i}-d )}   f_{a',i'}(n) f_{a,i} (n+d)  \gg | I_{a,i} \cap (I_{a',i'} + d)| \eps^2 \eta^2 \gg \eps^3 \eta^2 N/qM .  \]
Note that the right hand side here is much larger than the error terms in \eqref{newrev1}. Hence, the dominating contribution to the left hand side in \eqref{newrev1} is the term described above.

Amalgamating the above observations, we deduce that  
\begin{align} \label{ghkl1}
\Big| \{ d & \in I_{a-a',i-i'} :  \sum_{n \in I_{a',i'}}\1_{A_{a',i'}}(n )  \1_{A_{a,i}}(n+d) 
 \gg \varepsilon^3 \eta^2 |I_{a',i'}|\} \Big|  \nonumber \\
& \geq   \frac{N}{qM}\Big( \mu( \{ \vec{x} \in \TT^d : F_{a,i} \circ F_{a',i'}(\vec{x} ) \gg \varepsilon^2 \eta^2 \})  - O(\varepsilon^2 \eta^2+\eps+ 1/\cF(M)^{1/3}) \Big) \nonumber \\
& \geq   \frac{N}{qM}( \mu( \{ \vec{x} \in \TT^d : F_{a,i} \circ F_{a',i'}(\vec{x} ) \gg \varepsilon^2 \eta^2 \})  - O(\eps) ) ,
\end{align}
with the second and third error terms in the first step following from the exclusion of unsuitable choices of $d$ in the preceding discussion. 


\section{Applying Kneser and Inverse Kneser type inequalities} \label{invKnesersection}

We begin by recording some notation. Let $A, B \subseteq \mathbb{T}^d$. Then we define $A \Delta B$ to be the set $(A \setminus B) \cup (B \setminus A)$. Given bounded functions $f,g : \mathbb{T}^d \to \mathbb{C}$, we define the functions $f \circ g, f * g : \mathbb{T}^d \to \mathbb{C}$ by writing
\[ f \circ g (\vec{x}) = \int_{\mathbb{T}^d} f(\vec{x}) g(\vec{x} - \vec{y}) d \vec{y} \ \ \text{and} \ \ f * g (\vec{x}) = \int_{\mathbb{T}^d} f(\vec{y}) g(\vec{x} - \vec{y}) d \vec{y}  \]
for every $\vec{x} \in \mathbb{T}^d$. For every $(a,i) \in \mathbb{Z}/q\mathbb{Z} \times [M]$, we define the sets
\begin{align*}
    K_{a,i} & = \{ \vec{\gamma} \in \TT^d : F_{a,i}(\vec{\gamma}) \geq \eta/2\} ,\\ 
    S_{a,i} & = \{ \vec{\gamma} \in \TT^d : F_{a,i}(\vec{\gamma}) \geq \eta \}, \\
     T_{a,i} & = \{ \vec{\gamma} \in \TT^d : F_{a,i}(\vec{\gamma}) \geq 1- \eta^c \},
\end{align*} 
where $\eta, c$ are defined in \eqref{cot}. Note that $K_{a,i} \supseteq S_{a,i} \supseteq T_{a,i}$. We have
\[ \int_{\TT^d} F_{a,i}(\vec{\gamma}) d \vec{\gamma} \leq \mu(T_{a,i}) + (1 - \eta^c) \mu(S_{a,i}\setminus T_{a,i}) + \eta,  \]
and so, 
\begin{equation} \label{sai}
\mu(S_{a,i}) =  \mu(T_{a,i}) + \mu(S_{a,i} \setminus T_{a,i}) \geq \int_{\TT^d} F_{a,i}(\vec{\gamma}) d \vec{\gamma} + \eta^c \mu(S_{a,i} \setminus T_{a,i}) - \eta.   
\end{equation}
Here, as before, we define the measure $\mu$ on $\mathbb{T}^d$ to satisfy $\mu(S) = \nu_{d}( \pi^{-1}(S) \cap [0,1)^d)$ for every Borel set $S \subseteq \mathbb{T}^d$, where $\pi : \mathbb{R}^d \to \mathbb{T}^d$ is the standard projection map and $\nu_d$ is the Lebesgue measure on $\mathbb{R}^d$.

One may similarly argue that
\begin{equation} \label{hpl2}
    \mu(K_{a,i}) \geq \int_{\TT^d} F_{a,i}(\vec{\gamma}) d \vec{\gamma} + \eta^c \mu(K_{a,i} \setminus T_{a,i}) - \eta/2.  
\end{equation}

We will need the following analogue of Kneser's inequality in $\TT^d$ as proven by Tao \cite[Corollary 1.2]{Ta2018}.
\begin{lemma} \label{kneser}
    Let $S_1, S_2 \subseteq \TT^d$ be measurable sets. Let $\lambda \in \mathbb{R}$ satisfy 
    \[ 0 < \lambda < \min \{ \mu(S_1)^2, \mu(S_2)^2\}.\]
    Then 
    \[ \mu(\{\vec{x} \in \TT^d: \1_{S_1} * \1_{S_2}(\vec{x}) \geq \lambda \}) \geq \min\{1,\mu(S_1) + \mu(S_2)\} - O( \lambda^{1/2}). \]
\end{lemma}

Applying Lemma \ref{kneser}, we see that whenever $\mu(K_{a,i}), \mu(K_{a',i'})  > \lambda^{1/2}$, with $\lambda$ being as in \eqref{cot}, we have
\begin{align*} \label{ineq31}
   \mu(  \{\vec{x} \in \TT^d  : & F_{a,i}  \circ F_{a',i'} (\vec{x} ) \geq \lambda \eta^2/4 \} ) 
     \geq \mu(\{\vec{x} \in \TT^d :  \1_{K_{a,i}} \circ 1_{K_{a',i'}} (\vec{x} ) \geq \lambda \} )  \\
     & \geq \min\{1, \mu(K_{a,i} + \mu(K_{a',i'}) \}  - O(\lambda^{1/2}) \\
    & \geq \min\bigg\{1,\int_{\TT^d} F_{a,i} + \int_{\TT^d} F_{a',i'} + \eta^c ( \mu(K_{a,i} \setminus T_{a,i}) + \mu(K_{a',i'} \setminus T_{a',i'}) ) - \eta\bigg\} - O(\lambda^{1/2}),
\end{align*}
where the final step follows from \eqref{hpl2}. Since $\cF$ grows sufficiently fast in terms of $\varepsilon, \eta, \delta$, we may apply Lemma \ref{t1} to get that
\[ \int_{\TT^d} F_{a,i} + \int_{\TT^d} F_{a',i'} - \eta \geq \alpha(a,i) + \alpha(a',i') - O(\varepsilon + qM/\cF(M) + \eta) = \alpha(a,i) + \alpha(a',i') - O(\eta) \]
whenever $(a,i), (a',i') \notin E$.
 
Since we chose $\lambda = \varepsilon^2$, as defined in \eqref{cot}, we obtain the following characterisation. Given $(a,i), (a',i') \notin E$ satisfying $\mu(K_{a,i}), \mu(K_{a,i})  > \varepsilon$, we either have 
\begin{equation}  \label{con5}
\mu(K_{a,i}\setminus T_{a,i}),\mu(K_{a',i'} \setminus T_{a',i'})  < \delta^{1/2}  
\end{equation}
and
\begin{equation} \label{con6}
\mu(  \{\vec{x} \in \TT^d  : F_{a,i}  \circ F_{a',i'} (\vec{x} ) \geq \varepsilon^2 \eta^2/4 \} ) 
     \geq \min\{1,\alpha(a,i) + \alpha(a',i')\}  - O(\eta),
\end{equation}
or 
\begin{equation} \label{this}
\mu(  \{\vec{x} \in \TT^d  : F_{a,i}  \circ F_{a',i'} (\vec{x} ) \geq \varepsilon^2 \eta^2/4 \} ) \geq \min\{1 , \alpha(a,i) + \alpha(a',i') + \delta - \delta^{1000}\} - O(\varepsilon). 
\end{equation}

We will also require the following inverse version of Kneser's inequality due to Tao \cite[Theorem 1.3]{Ta2018}.

\begin{lemma} \label{tao}
Let $\varepsilon'>0$. Then for every $\delta>0$ sufficiently small in terms of $\varepsilon'$, if $S_1, S_2 \subseteq \TT^d$ satisfy 
\[ \mu(S_1), \mu(S_2), 1- \mu(S_1) - \mu(S_2) \geq \varepsilon' \] 
then either
\[ \mu( \{ \vec{x} \in \TT^d : \1_{S_1}*\1_{S_2}(\vec{x}) \geq \delta \} ) \geq \mu(S_1) + \mu(S_2) + \delta \]
or there exists some surjective homomorphism $\varphi : \TT^d \to \TT$ and some intervals $I_1, I_2 \subseteq \TT$ such that 
\[ \mu(S_1 \Delta \varphi^{-1}(I_1)), \mu(S_2 \Delta \varphi^{-1}(I_2)) \leq \varepsilon'. \]
\end{lemma}

 Now, let $(a_0,i_0) \in E$ be such that $\alpha(a_0,i_0)$ is maximal. Let 
\[ J = \{ (a,i) \notin E: \alpha(a,i) \geq 10 \varepsilon'\}. \]
We see that
\[ \sum_{(a,i) \in J} \alpha(a,i) (N/qM) + (qM) (10 \varepsilon')(N/qM)  + |E| (N/qM) \geq |A| , \]
whence, 
\[  \sum_{(a,i) \in J} \alpha(a,i) \geq ( \alpha- \varepsilon - 10 \varepsilon'  ) qM. \]
Firstly, this implies that 
\begin{equation} \label{smi}
    |J| \geq (\alpha - \varepsilon - 10 \varepsilon') qM.
\end{equation}
Moreover, since $|J| \leq qM$, we also get that 
\begin{equation} \label{tree1}
    \alpha(a_0,i_0) \geq \alpha - \varepsilon - 10 \varepsilon'. 
\end{equation}

Let $J'$ be the set of all $(b,j) \in \mathbb{Z}/q\mathbb{Z} \times [-M,M]$ such that $b= a - a'$ and $j = i-i'$ or $j =i - i'- 1$ for some $(a,i), (a',i') \in J$ with one of $(a,i), (a',i')$ equalling $(a_0, i_0)$. Note that $4|J| \geq |J'| \geq  |J|$. Furthermore, for any $(b,j) \in J'$, we see that
\begin{equation} \label{2min}
    \max_{\substack{(a,i), (a',i') \notin E, \\ a-a' = b \ \text{and} \ i-i' = j \ \text{or} \ j- 1}} (\alpha(a,i) + \alpha(a',i')) = \max_{\substack{(a,i),(a',i') \in J , \\  a-a' = b \ \text{and} \  i-i' = j \ \text{or} \ j- 1}} (\alpha(a,i) + \alpha(a',i')).
\end{equation}
Let $J''$ be the set of all $(a,i) \in J$ for which there exists some $(a',i') \in J$ such that $(a,i), (a',i')$ maximise the right hand side of \eqref{2min} for some $(b,j) \in J'$ and $\alpha(a,i) \geq \alpha(a',i')$. In particular, this means that whenever $(a,i) \in J''$, one has
\begin{equation} \label{dish}
\alpha(a,i) \geq \frac{\alpha(a,i) + \alpha(a',i)}{2} \geq \frac{\alpha(a_0,i_0)}{2} \geq \frac{ \alpha - \varepsilon - 10 \varepsilon'}{2},
\end{equation}
with the last inequality following from \eqref{tree1}.

Now, let $(a,i), (a',i') \in J$ satisfy 
\begin{equation} \label{assum3}
\mu(K_{a',i'} \setminus T_{a',i'}) , \mu(K_{a,i} \setminus T_{a,i}) < \delta^{1/2} \ \ \text{and} \ \ \mu(S_{a',i'}), \mu(S_{a,i}) < 1/2 - \varepsilon'. 
\end{equation}
Moreover, if $(a,i), (a',i')$ satisfy the maximum on the right hand side in \eqref{2min} for some $(b,j) \in J'$, then we must have
\begin{equation} \label{tree}
\max\{ \alpha(a,i), \alpha(a',i')\} \geq (\alpha(a,i) + \alpha(a',i'))/2 > \alpha(a_0,i_0)/2 \geq \alpha/2 - 6 \varepsilon',
\end{equation}
with the final step following from \eqref{tree1}. Since $J \cap E = \emptyset$, we can apply \eqref{sai} along with Lemma \ref{t1} to discern that
\[ \mu(S_{a',i'}) \geq \alpha(a',i') - O(\varepsilon - qM/\cF(M)) - \eta. \]
Recalling the choice of parameters in \eqref{cot} and noting the fact that $\delta$ is sufficiently small in terms of $\varepsilon'$, we get that 
\begin{equation} \label{shr}
\mu(S_{a',i'}) \geq \alpha(a',i') -  O(\varepsilon+ qM/\cF(M)) - \delta^{1000} \geq 10 \varepsilon' -\varepsilon' = 9 \varepsilon'.
\end{equation}
Applying Lemma \ref{tao}, we see that either
\begin{align} \label{tmne}
    \mu( \{ \vec{x} \in \TT^d : F_{a,i} \circ F_{a',i'}(\vec{x})&  \geq \delta \eta^2 \} )
     \geq
    \mu( \{ \vec{x} \in \TT^d :  \1_{S_{a,i}} \circ \1_{S_{a',i'}}(\vec{x}) \geq \delta \} )  \nonumber \\
    & \geq \mu(S_{a',i'}) + \mu(S_{a,i}) + \delta \nonumber \\
    & \geq \alpha(a',i')+ \alpha(a,i) + \delta -  O(\varepsilon+ qM/\cF(M)) - 2 \delta^{1000} 
\end{align}
or there exists some surjective homomorphism $\varphi : \TT^d \to \TT$ and some interval $I \subseteq \TT$ such that 
\begin{equation} \label{inne}
\mu(S_{a,i} \Delta \varphi^{-1}(I)) \leq \varepsilon'. \end{equation}

Now suppose that we are not in the case when \eqref{assum3} holds, but we have that
\begin{equation} \label{echo1}
\mu(K_{a',i'} \setminus T_{a',i'}) < \delta^{1/2} \ \ \text{and} \ \ \mu(S_{a',i'}) \geq 1/2 - \varepsilon'
\end{equation}
holds for either $(a,i)$ or $(a',i')$. In this case, we first note that
\[ (1 - \eta^c) \mu(S_{a',i'}) - \mu(S_{a',i'} \setminus T_{a',i'}) \leq (1 - \eta^c) \mu(T_{a',i'}) \leq \int_{\TT^d} F_{a',i'}(\vec{\gamma}) d \vec{\gamma} \leq \alpha(a',i') + O(\varepsilon),  \]
and so, we must have
\[( 1/2 - \varepsilon')(1 - \eta^c) \leq  (1 - \eta^c)\mu(S_{a',i'}) \leq  \alpha(a',i') + O(\varepsilon) + \delta^{1/2}. \]
The above implies that
\begin{equation} \label{fire}
\alpha(a',i') \geq 1/2 - \varepsilon' - 10 \delta^{1/2} - O(\varepsilon) .
\end{equation}

Finally, if neither \eqref{assum3} nor \eqref{echo1} holds for $(a',i')$ or $(a,i),$ then we must be in the situation when either 
\[ 
\mu(K_{a,i} \setminus T_{a,i}) \geq \delta^{1/2}  \ \ \text{or} \ \  \mu(K_{a',i'} \setminus T_{a',i'}) \geq \delta^{1/2}. 
\]
In particular, returning to the discussion surrounding \eqref{con5} and \eqref{this}, we see that in the above case, we must have \eqref{this} being true.

Thus, we now divide our proof of Theorem \ref{m1} into three cases:
\begin{enumerate}
    \item {\bf $1$-dimensional structured case}: Letting $C>0$ be some large absolute constant. There exists some $(a,i) \notin E$ such that
$\alpha(a,i) \geq 1/2 - C\varepsilon'$.
  \item {\bf $2$-dimensional structured case}: When \eqref{inne} and \eqref{assum3} hold for some $(a,i) \in J''$.
  \item {\bf Expansion Case}: When neither of the above two cases hold.
\end{enumerate}

Let us first handle the \emph{$1$-dimensional structured case}. In this case,
 we see that $|A \cap I_{a,i}| \geq (1/2 - O(\varepsilon')) |I_{a,i}|$, where $I_{a,i}$ is a $1$-dimensional arithmetic progression such that
\[ |I_{a,i}| \geq N/(qM) - O(1) \gg_{\varepsilon} N. \]
This precisely gives us the conclusion presented in \eqref{it2} in Theorem \ref{m1}. Hence, it suffices to study the \emph{Expansion Case} and the \emph{$2$-dimensional structured case}.


\section{Expansion case} \label{expansionsection}

Let $K$ be the set of all $(b,j) \in ( \mathbb{Z}/q\mathbb{Z} \times [-M,M]) \setminus J'$ such that there exist $(a,i), (a',i') \notin E$ satisfying $\alpha(a,i), \alpha(a',i') \geq \varepsilon$ and $(a,i) - (a',i') = (b,j)$ or  $(a,i) - (a',i') = (b,j-1)$. For every $(b,j) \in K \cup J'$, we fix some $(a_{(b,j)}, i_{(b,j)}), (a_{(b,j)}', i_{(b,j)}') \notin E$ such that
\[ a_{(b,j)} - a_{(b,j)}' = b \ \ \text{and} \ \ i_{(b,j)} - i_{(b,j)}' \in \{j,j-1\}  \ \ \text{and} \ \ \alpha(a_{(b,j)}, i_{(b,j)}) + \alpha(a_{(b,j)}', i_{(b,j)}') \ \text{is maximal}. \]
Given $\vec{b} \in J'$ and noting \eqref{2min}, we see that $\alpha(a_{\vec{b}}, i_{\vec{b}}), \alpha(a_{\vec{b}}', i_{\vec{b}}')  \geq 10 \varepsilon'$ and one of these is at least $\alpha/2 - \varepsilon/2 - 5 \varepsilon'$ since one of $(a_{\vec{b}}, i_{\vec{b}})$ or $(a_{\vec{b}}', i_{\vec{b}}')$ is an element of $J''$. Since we are not in the $1$-dimensional structured case or the $2$-dimensional structured case, we get either \eqref{this} or \eqref{tmne} to hold without the minimum against $1$ condition. In either case, we have
\begin{align} \label{reveqn1}
    \mu( \{ \vec{x} \in \TT^d : F_{a_{\vec{b}},i_{\vec{b}}} \circ F_{a_{\vec{b}}',i_{\vec{b}}'}(\vec{x}) & \gg \varepsilon^2 \eta^2 \} ) 
\geq \alpha(a_{\vec{b}}, i_{\vec{b}}) + \alpha(a_{\vec{b}}', i_{\vec{b}}') + \delta - 2\delta^{1000} - O(\varepsilon) \nonumber \\
& = \max_{\substack{a- a' = b, \\ i - i' \in \{j,j-1\}}}\{  \alpha(a,i) + \alpha(a',i') \} + \delta - 2\delta^{1000} - O(\varepsilon) .
\end{align}  
Now, let $\vec{b} \in K$. As before, since we are not in the $1$-dimensional structured case, we have either \eqref{con6} or \eqref{this} being true without the minimum against $1$ condition, and in either case, we have
\begin{align} \label{reveqn3}
       \mu( \{ \vec{x} \in \TT^d : F_{a_{\vec{b}},i_{\vec{b}}} \circ F_{a_{\vec{b}}',i_{\vec{b}}'}(\vec{x})  \geq \delta \eta^2 \} ) 
& \geq \alpha(a_{\vec{b}}, i_{\vec{b}}) + \alpha(a_{\vec{b}}', i_{\vec{b}}') - O(\eta) \nonumber \\
& = \max_{\substack{a- a' = b, \\ i - i' \in \{ j,j-1\}}}\{  \alpha(a,i) + \alpha(a',i') \}  - O(\eta) .
\end{align}

Let $\beta(a,i) = \alpha(a,i)$ whenever $(a,i) \notin E$ and $\alpha(a,i) \geq \varepsilon$, and $\beta(a,i) = 0$ whenever $(a,i) \in E$ or whenever $\alpha(a,i) < \varepsilon$. At this point, we need the following lemma due to Eberhard--Green--Manners \cite[Lemma 4.9]{EGM2014}. This is a nice application of the planar Brunn--Minkowski inequality. 

\begin{lemma} \label{egmbm}
    Let $\tilde{\eps} >0$. Let $\beta: \mathbb{Z}/q\mathbb{Z} \times [M] \to [0,1]$ such that for any $(a,i) \in \mathbb{Z}/q\mathbb{Z} \times [M]$, one has either $\beta(a,i) \geq \tilde{\varepsilon}$ or $\beta(a,i) = 0$.  Then 
    \[\sum_{(b, j) \in \mathbb{Z}/q\mathbb{Z}\times [-M,M]} \max_{\substack{a-a' = b,\\ i-i'
\in \{j,j-1\}}}\{  \beta(a,i) + \beta(a',i') \} \geq 4 \sum_{(a,i) \in \mathbb{Z}/q \mathbb{Z}\times [M]} \beta(a,i) - O( \tilde{\eps} qM). \]
\end{lemma}

We apply Lemma \ref{egmbm} to deduce that
\begin{align} \label{reveqn2}
\sum_{(b,j)\in K \cup J'} \max_{\substack{a-a' = b,\\ i-i'
\in \{j,j-1\}}}\{  \alpha(a,i) + \alpha(a',i') \} 
& = \sum_{(b, j) \in \mathbb{Z}/q\mathbb{Z}\times [-M,M]} \max_{\substack{a-a' = b,\\ i-i'
\in \{j,j-1\}}}\{  \beta(a,i) + \beta(a',i') \} \nonumber \\
& \geq 4 \sum_{(a,i) \in \mathbb{Z}/q\mathbb{Z}\times [M]} \beta(a,i) - O(\varepsilon qM) \nonumber \\
& \geq 4 \sum_{(a,i) \notin E} \alpha(a,i) - O(\varepsilon qM) \nonumber \\
& \geq 4 \sum_{(a,i) \in \mathbb{Z}/q\mathbb{Z}\times[M]} \alpha(a,i) - O(\varepsilon qM),
\end{align}
with the last inequality following from the fact that $|E| \leq \varepsilon^4 qM$, see Lemma \ref{exceptionlemma}.

Recalling  \eqref{ghkl1}, we see that
\begin{align*}
    | \{ d & \in [-N,N]:  \sum_{n \in [N]}\1_{A}(n )  \1_{A}(n+d) 
 \gg \varepsilon^3 \eta^2 N/(qM)\} | \\
 & \geq  \sum_{(b,j) \in K \cup J'} \max_{\substack{a-a' = b,\\ i-i'
\in \{j,j-1\}}} | \{ d  \in I_{b,j}:  \sum_{n \in I_{a',i'}}\1_{A_{a',i'}}(n )  \1_{A_{a,i}}(n+d) 
 \gg \varepsilon^3 \eta^2 |I_{a',i'}|\} |  \\
 & \geq   \frac{N}{qM} \sum_{(b,j) \in K \cup J'} \max_{\substack{a-a' = b,\\ i-i'
\in \{j,j-1\}}} ( \mu( \{ \vec{x} \in \TT^d : F_{a,i} \circ F_{a',i'}(\vec{x} ) \gg \varepsilon^2 \eta^2 \})  - O(\varepsilon) ).
\end{align*}
The contribution towards the right hand side above from terms attached to some $(b,j) \in J'$ is estimated in \eqref{reveqn1}. Similarly, the contribution of terms attached to some $(b,j) \in K$ is estimated in \eqref{reveqn3}. Adding these contributions together, we see that the  right hand side above is at least
\[ \bigg( \frac{N}{qM} \sum_{(b,j) \in K \cup J'} \max_{\substack{a-a' = b,\\ i-i'
 \in \{j,j-1\}}} \{  \alpha(a,i) + \alpha(a',i') \} \bigg) + \frac{N}{qM}|J'| ( \delta - 2\delta^{1000}) - O(\eta N + \eps N). \]
 We may now apply \eqref{reveqn2} and \eqref{smi} to deduce that this lower bound is 
\begin{align} \label{caseone}
     & \geq  4\frac{N}{qM}\sum_{(a,i) \in \mathbb{Z}/q\mathbb{Z}\times[M]}  \alpha(a,i) + \frac{N \delta\alpha}{2}   - O(\eta N + \eps N)  \nonumber \\
     & = 4 |A| + \frac{|A|\delta}{4} -  O(\eta N + \eps N)   \nonumber \\
     & \geq (4 + \delta/5) |A|.
\end{align}
Upon rescaling $\delta$, we obtain the conclusion presented in part \eqref{it1} of Theorem \ref{m1}. This concludes the Expansion case. 

\begin{remark} It is worth mentioning that it was quite crucial that all of the losses we incurred throughout our Kneser and inverse Kneser type steps were $O(\eta + \eps) = o(\delta)$. Furthermore, it was very important that we obtain the extra expansion factor $\delta - 2 \delta^{1000}$ not just for one $(b,j)$, but for $\gg \alpha qM$ many choices of $(b,j)$ since otherwise it could be the case that $M$ is so large in terms of $\delta$ that $N\delta/qM$ is much smaller than the aforementioned error terms $O(\eta N)$.
\end{remark}


\section{$2$-dimensional structured case: Controlling the projection} \label{structuresection}

In this section, we assume that we are in the $2$-dimensional structured case, that is, we assume that \eqref{inne} and \eqref{assum3} hold for some $(a,i) \in J''$. Recall that the former implies that there exists some surjective homomorphism $\varphi : \TT^d \to \TT$ and some interval $I \subseteq \TT$ and some $(a,i) \in J''$ such that 
\begin{equation} \label{quote2}
    \mu(S_{a,i} \Delta \varphi^{-1}(I)) \leq \varepsilon'.
\end{equation}
Let 
\[ B = \{ n \in I_{a, i} : n\vec{\theta} \in \varphi^{-1}(I) \} = \{  n \in I_{a, i} : n \varphi (\vec{\theta}) \in I \}. \]
We know that $\varphi(\vec{x}) = \vec{a} \cdot \vec{x}$ for all $\vec{x} \in \TT^d$, for some $\vec{a} \in \mathbb{Z}^d$ such that $\vec{a} \neq (0,\dots, 0)$. The aim of this section is to prove the following. 

\begin{Proposition}\label{prop: controlling |a|_infty}
    We have $\n{a}_{\infty} \leq \n{a}_{2} \ll_{M, \eta, \varepsilon', \delta} 1$.
\end{Proposition}

In this endeavour, we will first need to prove various preliminary lemmata. Thus, let $\nu_d$ be the Lebesgue measure on $\mathbb{R}^d$. We want to define a measure for sets $\mathcal{S}$ which are supported on finitely many translates of $\varphi^{-1}(0)$. By abuse of notation, we define, for every such Borel set $\mathcal{S}$, the measure
\[ \mu_{d-1}(\mathcal{S})  = \lim_{\delta \to 0} |\delta|^{-1} \mu(\mathcal{S}_{\delta})  \]
where $\mathcal{S}_{\delta}$ is the $\delta$-thickening of $\mathcal{S}$ along the normal direction $\vec{a}$. As $\mu$ is translation invariant in $\TT^d$, $\mu_{d-1}$ is also a translation invariant measure. 

Our proofs below will involve thickening the boundary of $\varphi^{-1}(I)$ and noting that half of this thickening should lie outside of $\varphi^{-1}(I)$. This can be justified by noting the fact that 
\begin{equation} \label{srt6}
  \mu_{\TT}(I) =   \mu(\varphi^{-1}(I)) < 1/2 + O(\varepsilon')
\end{equation}
holds; indeed combine \eqref{assum3} and \eqref{quote2}. Here, we use $\mu_{\TT}$ to denote the push forward of the Lebesgue measure on $\mathbb{R}$ to $\TT$. In fact, $\mu_{\TT}(I) =   \mu(\varphi^{-1}(I))$ holds because $\varphi$ is a surjective continuous group homomorphism and the push forward of the normalised Haar measure on $\TT^d$ is the normalised Haar measure on $\TT$.

Now, let $\mathcal{B}$ be the boundary of $\varphi^{-1}(I)$ and let $\Gamma(\vec{a}) = \mu_{d-1}(\mathcal{B})$.  
We want to prove the following result. 

\begin{lemma}\label{cor: Gamma(a) in T^d}
Under the above set up,    $\Gamma(\vec{a}) = 2\n{\vec{a}}_2$. 
\end{lemma}
\begin{proof}
    Without loss of generality we may assume that $I=(0,t)\subseteq \mathbb T$. 
    Clearly, 
    \[ \Gamma(\vec{a})=\mu_{d-1}(\varphi^{-1}(0)) + \mu_{d-1}(\varphi^{-1}(t))=2\mu_{d-1}(\varphi^{-1}(0)). \]
    We may further assume that $\mathrm{gcd}(a_1,\dots,a_d)=1$. Indeed, if $\mathrm{gcd}(a_1,\dots,a_d)=g>1$, then $\varphi$ factors through the map $\chi_g: \TT\to\TT$ satisfying $\chi_g(\vec{x}) = g\vec{x}$ for all $\vec{x} \in \mathbb{T}^d$, and the inverse of $\chi_g$ results in $g$ connected components. We then apply the same argument to each of the components and the claim still follows. Hence we may assume that $\mathrm{gcd}(a_1,\dots,a_d)=1$, which in particular implies that $\varphi^{-1}(0)$ is connected in $\TT^d$. 

    Now, note that $\Gamma(\vec{a}) = \mu_{d-1}(\varphi^{-1}(0)) + \mu_{d-1}(\varphi^{-1}(t)) = 2\mu_{d-1}(\varphi^{-1}(0))$, and so, it suffices to show that $\mu_{d-1}(\varphi^{-1}(0)) = \n{\vec{a}}_2$. Let $H = \{ \vec{x} \in \RR^d: \vec{a} \cdot \vec{x} =0 \}$. Then $\mu_{d-1}(\varphi^{-1}(0))$ equals the $(d-1)$-dimensional volume $\nu_{d-1}(D)$ of the fundamental domain $D$ of the lattice $\Lambda = \mathbb{Z}^d \cap H$. Since ${\rm gcd}(a_1, \dots, a_d) = 1$, there exists some $\vec{u} \in \mathbb{Z}^d$ such that $L + \mathbb{Z}\vec{u} = \mathbb{Z}^d$. The distance along the normal vector $\vec{a}$ between the hyperplanes $\vec{a} \cdot \vec{x} = 1$ and $\vec{a} \cdot \vec{x} = 0$ is $1/\n{\vec{a}}_2$. Writing $D'$ to be a fundamental domain of the lattice $\mathbb{Z}^d$, the above implies that
    \[ (1/\n{\vec{a}}_2) \nu_{d-1}(D) = \nu_{d}(D') = 1\]
    and so, we get that $\mu_{d-1}(\varphi^{-1}(0)) = \nu_{d-1}(D) = \n{\vec{a}}_2$. This concludes our proof of Lemma \ref{cor: Gamma(a) in T^d}. 
\end{proof}

In order to prove Proposition~\ref{prop: controlling |a|_infty}, 
it remains to show that $\Gamma(\vec{a}) \ll_{\eta, M, \varepsilon'} 1$. We will make use of the following lemma.

\begin{lemma}\label{lem: bound B'}
    Let
    \[ V_{a,i} =   \{ \vec{x} \in \mathbb{T}^d : F(\vec{x} ) = \eta \}, \]
    let $\mathcal{B}$ be the boundary of $\varphi^{-1}(I)$ in $\TT^d$ and let 
    \[
\mathcal{B}' = \{ \vec{b} \in \mathcal{B} : {\rm dist}(\vec{b}, V_{a,i}) \leq 100 \varepsilon'/ \Gamma(\vec{a}) \}.
\]
Suppose $|I|>100 \varepsilon'/ \Gamma(\vec{a})$. Then $\mathcal{B}'$ is $\mu_{d-1}$-measurable, and \[
\mu_{d-1}(\mathcal{B}') \geq \left(1-\frac{1}{25}\right) \mu_{d-1}(\mathcal B). 
\]
\end{lemma}
\begin{proof}

Observe that $\mathcal{B}'$ is Borel. Indeed, as $F_{a,i}$ is continuous, we get that $S_{a,i}$ is open and $V_{a,i}$ is closed. Moreover $\mathcal{B}$ is a closed manifold because $\varphi$ is continuous. Now since the distance function is continuous, and since $[0,100\varepsilon'/\Gamma(\vec{a})]$ is closed, we get that $\mathcal{B}'$ is closed in $\mathcal{B}$ with respect to the subspace topology, whence, $\mathcal{B}'$ is a Borel set. Let $r=100 \varepsilon'/ \Gamma(\vec{a})$ and let 
\[ \mathcal{C} = \mathcal{B}\setminus \mathcal{B}' = \{ \vec{b} \in \mathcal{B} : {\rm dist}(\vec{b}, V_{a,i}) > r\}.\]
We further denote $\mathcal{C}_1 = \mathcal{C} \setminus S_{a,i}$ and $\mathcal{C}_2 = \mathcal{C} \cap S_{a,i}$. Our aim now is to show that $\mu_{d-1}(\mathcal{C}_1)$ and $\mu_{d-1}(\mathcal{C}_2)$ are small.

Consider the $(r/2)$-thickening $\mathcal{C}_1'$ of $\mathcal{C}_1$ in both the normal directions to $\mathcal{B}$. 
By the quotient integral formula (integrating along normal lines to $\mathcal{B}$), we obtain $\mu(\mathcal{C}_1')= r\mu_{d-1}(\mathcal{C}_1).$ Moreover, we have
\[
\mathcal{C}_1' + B(r/2) \subseteq \mathcal{C}_1 + B(r),
\]
where $B(\tau)$ is a ball of radius $\tau>0$.  From the definition of $\mathcal{C}_1$, we have that $(\mathcal{C}_1+B(r))\cap S_{a,i}=\emptyset$. On the other hand, half of $\mathcal{C}_1 + B(r)$ is contained in $\varphi^{-1}(I)$ because $I$ is a closed interval. Let us call this set $X$. Our discussion implies that $X \subseteq \varphi^{-1}(I) \setminus S_{a,i}$, whence \eqref{quote2} implies that $\mu(X) \leq \varepsilon'$.  Thus, we have 
\[
r \mu_{d-1}(\mathcal{C}_1) \leq \mu(\mathcal{C}_1' + B(r/2))\leq  \mu(\mathcal{C}_1 + B(r)) = 2\mu(X) \leq 2\varepsilon',
\]

We can argue similarly to show that $r\mu_{d-1}(\mathcal{C}_2) \leq 2 \varepsilon'$. The key points of difference to note here are as follows. Writing $\mathcal{C}_2'$ to be the $r/2$-thickening of $\mathcal{C}_2$ in both the normal directions to $\mathcal{B}$, we can show that half of $\mathcal{C}_2 + B(r)$ must lie completely inside $S_{a,i}$ while being outside $\varphi^{-1}(I)$. On the other hand, we know that $\mu(S_{a,i} \Delta \varphi^{-1}(I)) \leq \varepsilon'$. As before, this means that 
\[ r \mu_{d-1}(\mathcal{C}_2) \leq \mu(\mathcal{C}_2' + B(r/2)) \leq \mu(\mathcal{C}_2 + B(r)) \leq 2 \varepsilon' . \]

Combining the two upper bounds, we see that
\[ \mu_{d-1}(\mathcal{C}) \leq  \mu_{d-1}(\mathcal{C}_1) + \mu_{d-1}(\mathcal{C}_2) \leq 4 r^{-1} \varepsilon' = \Gamma(\vec{a})/25.   \]
We conclude our proof by noting that $\mu_{d-1}(\mathcal{B}') \geq \mu_{d-1}(\mathcal{B}) - \mu_{d-1}(\mathcal{C})$.
\end{proof}

Let us now prove the main result of the section.

\begin{proof}[Proof of Proposition~\ref{prop: controlling |a|_infty}]
By Lemma~\ref{cor: Gamma(a) in T^d}, it suffices to show $\Gamma({\vec a})\ll_{\eta,M,\eps'}1$.  Now, thicken $V_{a,i}$ by $\eta/(10^3M)$, and note that due to $F$ being $M$-Lipschitz, these new elements lie in the set
\[ X' = \{ \vec{x} : 11 \eta/10\geq  F(\vec{x}) \geq 9\eta/10 \} \subseteq K_{a,i} \setminus T_{a,i}.\] Since we are in the $2$-dimensional structured case,  we know that \eqref{assum3} holds, which gives us  $\mu(X') \leq \mu(K_{a,i} \setminus T_{a,i}) < \delta^{1/2}$.

Now if $\Gamma(\vec{a}) \leq 10^8 M \varepsilon'/ \eta$, then we would be done, whence assume otherwise. This means that $100 \varepsilon'/\Gamma(\vec{a}) < \eta/(10^6 M)$. Thus, since elements of $\mathcal{B}'$ are at a distance of at most $100\varepsilon'/\Gamma(\vec{a})$ from the set $V_{a,i}$, the $\eta/(10^8 M)$-thickening of $\mathcal{B}'$ must lie in the $\eta/(10^3M)$-thickening of $V_{a,i}$ which itself is a subset of $X'$. This means that
\[  \mu_{d-1}(\mathcal{B}') (\eta/(10^8 M) ) \leq \mu(X') < \delta^{1/2}.\]
By Lemma~\ref{lem: bound B'}, the above inequality tells us that 
\[
\Gamma(\vec{a}) \ll M \delta^{1/2}\eta^{-1}\ll_{\eta,M,\eps'}1,
\]
where in the last inequality we use the definition that $\eta=\delta^{1000}$. Thus we are done.
\end{proof}


\section{$2$-dimensional structured case: Lipschitz approximation of $\varphi^{-1}(I)$ and inhomogeneous Bohr sets} \label{bohrsection}

Let $B$ be the \emph{inhomogeneous Bohr set}
\begin{equation}\label{defb}
    B = \{ n \in I_{a,i} : n \vec{\theta} \in \varphi^{-1}(I) \} = \{ n \in I_{a,i} : n (\vec{a} \cdot \vec{\theta} ) \in I \}. 
\end{equation}
Our main aim in this section is to prove that
\[ |B| \gg \alpha N/(qM) \ \ \text{and} \ \ |A \cap B|/|B| \geq 1 - O(\varepsilon'). \]

We begin by recording a lower bound for $\mu(\varphi^{-1}(I))$. 

\begin{lemma} \label{lem: lower bound on phi^{-1}(I)}
    Let $(a,i)\in J''$. We have $\mu(\varphi^{-1}(I))\geq \alpha/2 - 16 \varepsilon'$. 
\end{lemma}
\begin{proof}
    By \ref{inne}, we have $\mu(\varphi^{-1}(I))\geq \mu(S_{a,i})-\eps'$. By combining this with \eqref{shr}, we get that 
    \[
    \mu(S_{a,i})\geq \alpha(a,i)-O(\eps-qM/\mathcal{F}(M))-\eta \geq \alpha(a,i) - 10\eps',
    \]
    with the last inequality following from the choice of parameters fixed in \eqref{cot} and the fact that $\cF$ grows sufficiently rapidly. As $(a,i)\in J''$, \eqref{tree} implies that $\alpha(a,i)\geq (\alpha(a,i) + \alpha(a',i'))/2\geq \alpha/2-6\eps'$. Combine everything together gives us the desired conclusion.  
\end{proof}

We will now provide a Lipschitz approximation for $\1_{\varphi^{-1}}(I)$.

\begin{lemma}\label{lem: making f_i Lipschitz}
Let $\sigma >0$ be small enough in terms of $\alpha$. Then there exist Lipschitz functions $f_1, f_2$ which satisfy 
\begin{equation} \label{Lpapprox}
f_1(\vec{x}) \leq \1_{\varphi^{-1}(I)} (\vec{x}) \leq f_2(\vec{x}) \ \text{for all} \ \vec{x} \in \mathbb{T}^d, \ \  \text{and} \ \ \int_{\TT^d}  (f_2(\vec{x}) - f_1(\vec{x}) ) d \vec{x} \leq \sigma, 
\end{equation}
and both $f_1,f_2$ are $O_{M, \eta, \delta,\eps', \sigma}(1)$-Lipschitz.
\end{lemma}
\begin{proof}
Let $I$ be the projection of the interval $[a,b]$ to $\mathbb{T}$, for some $a,b \in \RR$ satisfying $a < b$ and $a - b < 1/2 + O(\varepsilon')$, with the latter observation following from \eqref{srt6}. Take $\tau=\sigma/4$ and let $B_1(\tau)$ be the projection of $[-\tau/2, \tau/2]$ to $\mathbb{T}$. Define $I_1$ to be the projection of the interval $[a- \tau, b+\tau]$ to $\mathbb{T}$, and $I_2=I+B_1(\tau)$.  Define two functions $g_1,g_2:\TT\to[0,1]$
   \[
   g_1=\frac{1}{\tau}\1_{I_1}*\1_{B_1(\tau)} \quad\text{and }\quad g_2=\frac{1}{\tau}\1_{I_2}*\1_{B_1(\tau)}. 
   \]
and $f_1,f_2:\TT^d\to [0,1]$ by $f_1(\vec{x})=g_1(\varphi(\vec{x}))$ and $f_2(\vec{x})=g_2(\varphi(\vec{x}))$, for every $\vec{x}\in\TT^d$. 

Note that $I_1+B_1(\tau)\subseteq I$, and so, for any $\vec{x} \in \TT^d$ satisfying $\varphi(\vec{x})\notin I$, one has $\varphi(\vec{x})-u\notin I_1$ for every $u \in B_1(\tau)$. Thus, for $\vec{x}$ as above, we get that 
\[
f_1(\vec{x}) = g_1(\varphi(\vec{x}))=\frac{1}{\tau}\int_{B_1(\tau)} \1_{I_1}(\varphi(\vec{x})-u)du =0,
\]
whence $f_1\leq \1_{\varphi^{-1}(I)}$.  Similarly, $I\subseteq I_2 -B_1(\tau)$, and so, whenever $\vec{x} \in \TT^d$ satisfies $\phi(\vec{x})\in I$, then $\phi(\vec{x})-u\in I_2$ for every $u\in B_1(\tau)$. Therefore for any such $\vec{x}$, one has
\[
f_2(\vec{x}) = g_2(\varphi(\vec{x}))=\frac{1}{\tau}\int_{B_1(\tau)} \1_{I_2}(\varphi(\vec{x})-u)du =1,
\]
which gives $f_2\geq \1_{\varphi^{-1}(I)}$. As before, since $\varphi$ is a surjective continuous group homomorphism  and the push forward of the normalised Haar measure on $\TT^d$ is the normalised Haar measure on $\TT$, we get that
\[
\int_{\TT^d} (f_2(\vec{x})-f_1(\vec{x}))d \vec{x} = \int_\TT (g_2(x)-g_1(x)) dx = \mu_\TT(I_2) - \mu_\TT(I_1)\leq 4\tau = \sigma. 
\]

It remains to bound the Lipschitz constants for $f_1$ and $f_2$. Let us first bound the Lipschitz constants for $g_1$ and $g_2$. Note that for any $x,y\in\TT$, one has
\begin{align*}
|g_1(x+y)-g_1(x)|&=\frac{1}{\tau}\Big|\int \1_{I_1}(u)\big(\1_{B_1(\tau)}(x+y-u)-\1_{B_1(\tau)}(x-u)\big)du\Big|\\
&\leq \frac{1}{\tau} \big|\mu_{\TT}\big((B_1(\tau)+y)\triangle B_1(\tau)\big)\big|\leq \frac{2}{\tau}|y|. 
\end{align*}
The same argument works for $g_2$ with $I_1$ replaced by $I_2$, hence both $g_1, g_2$ are $2/\tau$-Lipschitz.  Finally, as $\varphi(\vec{x})=\vec{a}\cdot x$, for every $\vec{x}\in\TT^d$, one has
\[
|\varphi(\vec{x})-\varphi(\vec{y})|\leq \|\vec{a}\|_2\|\vec{x}-\vec{y}\|.
\]
Therefore for any $i \in \{1,2\}$ and for any $\vec{x}, \vec{y} \in \TT^d$, we have
\[
|f_i(\vec{x})-f_i(\vec{y})| = |g_i(\varphi(\vec{x}))-g_i(\varphi(\vec{y}))|\leq\frac{2}{\tau}|\varphi(\vec{x})-\varphi(\vec{y})|\leq \frac{2\|\vec{a}\|_2}{\tau}\|\vec{x}-\vec{y}\|.
\]
This implies that both $f_1$ and $f_2$ are $(8\|\vec{a}\|_2/\sigma)$-Lipschitz. The desired conclusion now follows from  Proposition~\ref{prop: controlling |a|_infty}. 
\end{proof}

We will need one more lemma which implies that functions with a small $U^2$ norm are almost orthogonal to pullbacks of Lipschitz functions.

\begin{lemma}\label{gowers-orthog-struct}
Suppose that $d, M \in \NN$ and that $\delta > 0$. Then for some $\delta_*=\delta_*(d, M,\delta)>0$ and all sufficiently large $N \geq N_0(d,M,\delta)$ the following is true. Let $f(n) = F(n\vec{\theta})$ for all $n \in [N]$, where $F:\TT^d\to[0,1]$ is $M$-Lipschitz and $\vec{\theta} \in \TT^d$, and suppose $g : [N] \to [-1,1]$ satisfies $\|g\|_{U^2} \leq \delta_*$. Moreover, suppose $P$ is an arithmetic progression in $[N]$ with $|P| \geq N/M^2$. Then 
\begin{equation} \label{unf38}
\left| \frac{1}{|P|} \sum_{n \in P} f(n) g(n) \right| \leq \delta.
\end{equation}
\end{lemma}

\begin{proof}
    Since $F$ is $M$-Lipschitz, we can use \cite[Lemma A.9]{GT2008} to find some $M_0 \ll_{d, M, \delta} 1$ such that
    \[ F(\vec{x}) = \sum_{\substack{ \vec{m} \in \mathbb{Z}^d, \\ \n{\vec{m}}_1 \leq M_0}} c_{\vec{m}} e(\vec{m} \cdot \vec{x}) + O(\delta) \]
    holds uniformly for all $\vec{x} \in \TT^d$, with $|c_{\vec{m}}| \ll_{M,d} 1$ for all $\vec{m} \in \mathbb{Z}^d$ with $\n{\vec{m}}_1 \leq M_0$. Thus, it suffices to show that 
    \[ |N^{-1} \sum_{n \in [N]} (e(\beta n) \1_{P}(n))  g(n) |  \leq \delta \]
holds uniformly for all $\beta \in \TT$. The idea is to consider this sum in $\mathbb{Z}/p\mathbb{Z}$ for some prime $p \in [200N, 400N)$. Thus, writing $f_1(x) = e(\beta x)\1_P(x)$ for all $x\in \mathbb{Z}/p\mathbb{Z}$, one has
\[ p^{-1}\sum_{n \in \mathbb{Z}/p\mathbb{Z}} f_1(x) g(x) =  \sum_{r\in \mathbb{Z}/p\mathbb{Z}} \hat{f_1}(r) \hat{g}(r), \]
where $\hat{h}(r) = p^{-1} \sum_{x\in \mathbb{Z}/p\mathbb{Z}} h(x) e(-xr/p)$ for every $h : \mathbb{Z}/p \mathbb{Z} \to \mathbb{C}$. Applying H\"{o}lder's inequality, we get that
\[ \sum_{r\in \mathbb{Z}/p\mathbb{Z}} \hat{f_1}(r) \hat{g}(r) \leq (\sum_{r \in \mathbb{Z}/p\mathbb{Z}} |\hat{f_1}(r)|^{4/3} )^{3/4} (\sum_{r \in \mathbb{Z}/p\mathbb{Z}} |\hat{g}(r)|^{4} )^{1/4} .\]
Using \eqref{equivdefuniform}, we see that the second term on the right hand side is $O(\n{g}_{U^2}^{1/4})$. Thus, it suffices to show that the first term on the right hand side is $O(1)$. In particular, writing $P = a  + q\cdot [N'] \subseteq [N]$, we see that
\begin{align*}
    \hat{f_1}(r)  = p^{-1} e(a(\beta - r/p)) \sum_{n=1}^{N'} e(qn(\beta - r/p))  ,
\end{align*}
and so,
\[ |\hat{f_1}(r)| \ll p^{-1}\min\left\{ \frac{1}{ \n{q(\beta - r/p)}}_{\TT}, N' \right\} .\]
Thus, we get that
\[ \sum_{r \in \mathbb{Z}/p\mathbb{Z}} |\hat{f_1}(r)|^{4/3} \ll 1 + p^{-4/3} \sum_{r' = 2}^{p/2} \frac{1}{ \n{r'/p}_{\TT}^{4/3}} \ll 1. \]
Putting everything together gives the claimed estimate.
    \end{proof}

We now return to our main aim of showing that the inhomogeneous Bohr set $B$ described in \eqref{defb} is dense in $[N]$ and the set $A$ has density $1 - o(1)$ on $B$. We begin by applying Lemma \ref{lem: making f_i Lipschitz} with 
\begin{equation} \label{sigmadef}
\sigma = \varepsilon.
\end{equation}
Thus, writing $f_1,f_2$ to satisfy \eqref{Lpapprox} and using Lemma~\ref{distribution-integral-a}, we get that
\begin{align} \label{whisp}
    |I_{a, i}| \left( \int_{\TT^d}  f_1(\vec{\gamma})  d\vec{\gamma} - \varepsilon \right)  &  \leq \sum_{n \in I_{a, i}} f_1(n \vec{\theta})  
    \leq \sum_{n \in I_{a, i}} \1_{n \vec{\theta} \in \varphi^{-1}(I)} \nonumber \\
    & \leq \sum_{n \in I_{a, i}} f_2(n \vec{\theta})  \leq |I_{a, i}| \left(\int_{\TT^d}  f_2(\vec{\gamma})   d \vec{\gamma}+ \varepsilon \right) .
\end{align} 
Moreover, we have 
\[ \int_{\TT^d} f_2(\vec{\gamma}) d\vec{\gamma} - \mu(\varphi^{-1}(I)) < \sigma  \ \ \text{and} \ \ \mu(\varphi^{-1}(I)) - \int_{\TT^d} f_1(\vec{\gamma}) d\vec{\gamma} < \sigma. \]
Recalling \eqref{defb}, we can now use \eqref{Lpapprox} and \eqref{whisp} to prove that
\begin{equation} \label{coll}
    |B| = \sum_{n \in I_{a, i}} \1_{n \vec{\theta} \in \varphi^{-1}(I)} = |I_{a, i}| ( \mu(\varphi^{-1}(I)) + O( \varepsilon + \sigma ) )
\end{equation}

Next, we will show the following lemma. 

\begin{lemma} \label{beach}
With the set $B$ defined as in \eqref{defb}, we have
 \[ |A \cap B| = \sum_{n \in I_{a,i}} \1_A( n) \1_{n \vec{\theta} \in \varphi^{-1}(I)} = |I_{a,i}| ( \mu( \varphi^{-1}(I))  + O(\varepsilon' + \sigma + 1/\cF(M) ).\]
 \end{lemma}

 \begin{proof}
As in the preceding discussion, it suffices to show that 
\begin{equation} \label{bnw}
|\sum_{n \in I_{a,i} } f_1 (n \vec{\theta}) \1_A(n)  - |I_{a,i}| \mu(\varphi^{-1}(I))| \ll  |I_{a,i}|( \varepsilon' + \sigma)  
\end{equation}
and
\[ |\sum_{n \in I_{a,i} } f_2 (n \vec{\theta})\1_A(n) - |I_{a,i}|\mu(\varphi^{-1}(I)) |  \ll  |I_{a,i}|( \varepsilon' + \sigma)  .\]
We will prove the first inequality and the second will follow in a similar manner.

Noting \eqref{arl1} and \eqref{l2sml}, we see that it suffices to show that the inequalities
 \begin{equation} \label{bnw2}
|\sum_{n \in I_{a,i}} F_{a,i}(n \vec{\theta}) f_1(n \vec{\theta}) -  |I_{a,i}| \mu(\varphi^{-1}(I_{a,i}))| \ll |I_{a,i}|( \varepsilon + \eta + \delta^{1/2} + \sigma + \varepsilon') \end{equation}
and
\begin{equation} \label{bnw3}
|\sum_{n \in I_{a,i}} f_{\rm unf}(n) f_1(n \vec{\theta}) | \ll \eps |I_{a,i}|
\end{equation}   
hold. We see that \eqref{bnw3} follows in a straightforward manner from Lemma~\ref{gowers-orthog-struct}, and so, we turn to proving the inequality recorded in \eqref{bnw2}. Combining Lemma~\ref{distribution-integral-a} along with the fact that product of two $O(1)$-bounded functions, each with Lipschitz constants $\ll_{\eta, M, \delta}1$, is also $O(1)$-bounded and has its Lipschitz constant $\ll_{\eta, M , \delta}1$, one finds that
\[  
\Big|\sum_{n \in I_{a,i}} F_{a,i}(n \vec{\theta}) f_1(n \vec{\theta}) -  |I_{a,i}| \int_{} F_{a,i} (\vec{\gamma}) f_1(\vec{\gamma}) d \vec{\gamma} \Big| < \varepsilon|I_{a,i}| . 
\]
Note that
\begin{align*}
    \int  F_{a,i} (\vec{\gamma})f_1(\vec{\gamma})d \vec{\gamma}
    & = \int F_{a,i} (\vec{\gamma})\1_{\varphi^{-1}(I)}(\vec{\gamma})d \vec{\gamma} + \int F_{a,i}(\vec{\gamma}) ( f_1(\vec{\gamma}) - \1_{\varphi^{-1}(I)}(\vec{\gamma}))d \vec{\gamma} \\
    & = \int \1_{\varphi^{-1}(I)}(\vec{\gamma})d \vec{\gamma} + O(\varepsilon' + \eta^c + \delta^{1/2}) + O\left( \int ( \1_{\varphi^{-1}(I)}(\vec{\gamma}) - f_1(\vec{\gamma}))d \vec{\gamma} \right) \\
    & = \mu(\varphi^{-1}(I))+ O(\varepsilon' + \eta + \delta^{1/2}) + O(\sigma) ,
\end{align*}
where in the second step, we have used the facts that 
\[ \mu(S_{a,i} \Delta \varphi^{-1}(I)) \leq \varepsilon',\ \ \  \text{and} \ \  \mu(S_{a,i} \setminus T_{a,i}) \leq \mu(K_{a,i} \setminus T_{a,i}) \leq \delta^{1/2},\]
and $F_{a,i}$ is $1$-bounded, and $1 - \eta^c \leq F_{a,i}(\vec{\gamma}) \leq 1$ for all $\vec{\gamma} \in T_{a,i}$, and $f_1(\vec{\gamma}) \leq \1_{\varphi^{-1}(I)}(\vec{\gamma})$ for all $\vec{\gamma} \in \TT^d$. This combines with the preceding inequality to deliver \eqref{bnw2}.
\end{proof}

Note that combining \eqref{coll} and Lemma \ref{beach} implies that
\[ |A\cap B| = |B| + O( (\varepsilon' + \sigma)|I_{a,i}|).   \]
Putting together Lemma~\ref{lem: lower bound on phi^{-1}(I)} and \eqref{sigmadef} along with the fact that $\varepsilon'$ is sufficiently small in terms of $\alpha$, we discern that
\[ |A \cap B|/|B| \geq 1- O(\varepsilon'). \]
Let $\theta = \vec{a} \cdot \vec{\theta}$. Since $\n{\vec{a}}_2 \ll_{M, \eta, \delta} 1$ and since $\vec{\theta}$ is $(\cF(M),N)$-irrational, we must have that ${\theta}$ is $(c\cF(M),N)$-irrational, where $c>0$ is some positive constant depending on $M, \eta, \delta$.

\section{$2$-dimensional structured case: From Bohr sets to progressions} \label{continuedfracsection}

Recall that a $2$-dimensional progression $P$ is a set of the form
\[ P = \{a_0 + l_1 a_1 + l_2 a_2 : \ l_i \in [L_i] \ \  (1 \leq i \leq 2) \},\]
where $L_1, L_2$ are positive integers and $a_0, a_1, a_2 \in \mathbb{Z}$. Moreover $P$ is proper if $|P| = L_1L_2$.  The main result that we want to prove in this section is as follows.

\begin{lemma} \label{bohrprog}
There exists a proper $d$-dimensional progression $Q \subseteq B$ such that $|Q| \gg |B|$ and $d \leq 2$.
\end{lemma}

A well-studied theme in additive combinatorics is that low-rank homogeneous Bohr sets contain large low-dimensional progressions, see for example \cite[\S4.4]{TV2006}. Most of these results usually rely on some nice argument from geometry of numbers. For our purposes, we want to show that the inhomogeneous Bohr set $B$ described in \cite{Ta2012} contains a large $2$-dimensional progression.  An important piece of information that we will use here is that $\theta = \vec{a}\cdot \vec{\theta}$ is $(c \cF(M),N)$-irrational, where $c>0$ is some positive constant depending on $M, \eta, \delta$. This allows us to find a large translate of a homogeneous Bohr set $B''$ inside our inhomogeneous Bohr set $B$. 

At this point, we observe that if $\theta \in \mathbb{T} \setminus \mathbb{Q}$, then we can use the classical theory of continued fractions to find a large $2$-dimensional progression inside $B''$, see \cite{Ta2012} for a very nice discussion by Tao on this topic. On the other hand, $(c\cF(M), N)$-irrationality of $\theta$ does not actually imply that $\theta \in \mathbb{T} \setminus \mathbb{Q}$, and  the continued fraction strategy does not seem to work in a straightforward manner for the case when $\theta \in \mathbb{Q}$ since the continued fraction of $\theta$ may terminate too quickly. Thus, in this case, we use  classical tools from geometry of numbers to prove our desired claim. This is recorded in the following lemma.

\begin{lemma} \label{geomnumber}
Let $\alpha \in (0,1)$ and let $\sigma \in (0,1/100)$. Let $N$ be sufficiently large in terms of $\sigma$. Then the set 
\begin{equation} \label{alphasigma}
    B_N(\alpha, \sigma) = \{ n \in \mathbb{Z} : |n/N| \leq 1 \ \text{and} \   \n{ n \alpha}_{\mathbb{T}} < \sigma  \} 
\end{equation}
contains a proper $d$-dimensional progression $P$ such that $|P| \gg \sigma N$ and $d \leq 2$.
\end{lemma}

\begin{proof}
We first consider the case when $\alpha \in (0,1) \setminus \mathbb{Q}$. In this setting, we use standard results about Bohr sets of the form $B_N(\alpha, \sigma)$ and continued fraction decompositions of $\theta$ to deduce that $B_N(\alpha, \sigma)$ contains a proper $2$-dimensional progression $P'$ such that $|P'| \gg \sigma N$, see this post of Tao \cite[\S2]{Ta2012}. In particular, the results in \cite[\S2]{Ta2012} imply that one can find a set $P' \subseteq B_N(\alpha, \sigma)$ satisfying
 \[ P' = \{ m_1 q_1 + m_2 q_2 \ : \ m_1, m_2 \in \ZZ \ \text{and} \ |m_1| < M_1  \ \text{and}  \ |m_2| < M_2 \}, \]
 where $M_1, M_2, q_1, q_2$ are positive integers such that $q_1, q_2$ are coprime and $M_1 < q_2/100$ and $M_1 M_2 \gg \sigma N$. The first two conditions ensure that the progression $P'$ is proper and the last condition implies that $|P'| \gg \sigma N$.

We may thus assume that $\alpha = u/v$ where $1 \leq u < v$ are coprime integers.    Consider the lattice $\Lambda \subseteq \mathbb{Z}^2$ given by
    \[ \Lambda = \{(n, un + kv) : n, k \in \mathbb{Z}\}  = \mathbb{Z} \cdot (1,u) + \mathbb{Z} \cdot (0,v). \] Note that $\det(\Lambda) = v$. Let 
    \[ K = \{(z_1, z_2) \in \mathbb{R}^2 : |z_1| \leq N \ \text{and} \ |z_2| \leq \sigma v\}\] 
    and let $0 < \lambda_1\leq \lambda_2$ be the successive minima of $K$ with respect to $\Lambda.$ By Minkowski's second theorem, we get that
    \begin{equation} \label{minkowskisecond}
    \lambda_1 \lambda_2 \leq \frac{4 \det(\Lambda)}{(2N)(2\sigma v)}  = \frac{1}{\sigma N}.  
    \end{equation}
    Combining the definition of successive minima, the fact that $K$ is compact along with a standard compactness argument, we can prove that there exist linearly independent vectors $\vec{b}_1, \vec{b}_2 \in \mathbb{R}^2$ such that $\vec{b}_1 \in \Lambda \cap (\lambda_1 \cdot K)$ and $\vec{b}_2 \in \Lambda \cap (\lambda_2 \cdot K)$. For $i \in \{1,2\}$, let $\vec{b}_i = (x_i, y_i)$. Then we have that
    \begin{equation} \label{congruencecondition}
    |x_i| \leq \lambda_i N \ \ \text{and} \ \ |y_i| \leq \lambda_i \sigma v   \ \ \text{and} \ \ ux_i \equiv y_i \ ({\rm mod} \ v)
    \end{equation}
    for every $i \in \{1,2\}$. Now, set
    \[ L_1 = \lfloor 1/(10 \lambda_1) \rfloor \ \ \text{and} \ \  L_2 = \lfloor 1/(10 \lambda_2) \rfloor \ \ \text{and} \ \ P = \{ l_1 x_1 + l_2 x_2 : |l_1| \leq L_1 \ \text{and} \ |l_2| \leq L_2 \}. \]
    Since $\lambda_1 \leq \lambda_2$, inequality \eqref{minkowskisecond} implies that $\lambda_1 \leq 1/(\sigma N)^{1/2}$, whence $1/(10 \lambda_1) >1$. Thus $P$ is non-empty. 
    
    Note that $P \subset [-N,N]$ since for any admissible $l_1, l_2$, we may use \eqref{congruencecondition} to see that 
    \[ |l_1 x_1 + l_2 x_2| \leq  \frac{|x_1|}{10 \lambda_1} + \frac{|x_2|}{10 \lambda_2} \leq \frac{N}{5}.  \] 
    Moreover, for any such $l_1, l_2$, we use \eqref{congruencecondition} to see that
    \[ \left\|\frac{u(l_1x_1 + l_2 x_2)}{v}  \right\|_{\mathbb{T}} =  \left\| \frac{l_1 y_1 + l_2 y_2}{v} \right\|_{\mathbb{T}} \leq \frac{\n{ \frac{y_1}{v} }_{\mathbb{T}}}{10 \lambda_1} + \frac{\n{ \frac{y_2}{v} }_{\mathbb{T}}}{10 \lambda_2} \leq \sigma/5. \]
    Thus $P \subseteq B_N(\alpha, \sigma)$.  

    If $\lambda_2 \geq 1/10$, then \eqref{minkowskisecond} implies that $1/\lambda_1 \geq \sigma N/10$, and so $P$ contains a $1$-dimensional sub-progression of length $\gg 1/\lambda_1 \gg \sigma N/10$. Moreover, by definition, a $1$-dimensional progression is always proper. Thus we are done in this case.
    
    Now, suppose $\lambda_2 < 1/10$, in which case, $P$ is $2$-dimensional. We will show that $P$ is proper by contradiction, and so, suppose
    \[  l_1 x_1 + l_2 x_2 = n_1 x_1 + n_2 x_2 \]
    for some admissible $l_1, l_2, n_1, n_2$ satisfying $l_1 \neq n_1$ and $l_2 \neq n_2$. In this case, 
    \[ (l_1 - n_1) u x_1 + (l_2 - n_2) u x_2 \equiv 0 \ {\rm (mod} \ v) \] 
    and so, \eqref{congruencecondition} implies that
    \[ (l_1 - n_1) y_1 + (l_2 - n_2) y_2 \equiv 0 \ {\rm (mod} \ v).  \]
    Since $|l_i - n_i| \leq  1/(5 \lambda_i)$ and $|y_i| < \lambda_i \sigma v$ and $\sigma < 1/100$, we see that the above congruence holds if and only if $(l_1 - n_1)y_1 + (l_2 - n_2) y_2 = 0$. This means that 
    \[ (l_1 - n_1) \cdot \vec{b}_1 + (l_2 - n_2) \cdot \vec{b}_2 = 0, \]
    wherein, the linear independence of $\vec{b}_1, \vec{b}_2$ forces $l_1 = n_1$ and $l_2 = n_2$, delivering the desired contradiction. Finally, since $P$ is proper and $\lambda_2 < 1/10$, we use \eqref{minkowskisecond} to get that
    \[ |P| \gg \frac{1}{\lambda_1 \lambda_2} \gg \sigma N. \]
    This concludes our proof of Lemma \ref{geomnumber}.
\end{proof}

We are now ready to prove Lemma \ref{bohrprog}.

\begin{proof}[Proof of Lemma \ref{bohrprog}]
Note that $qn + a \in B$ for some $n \in [N/qM]$ if $(qn + a) \theta \in I$, that is, if $ n (q \theta) \in I - a\theta$. Moreover since $q \theta$ is $(c\cF(M)/M, N)$-irrational, it suffices to consider the case when $q=1$ and $a=0$, that is, when $I_{a,i} = [N/M]$. Similarly, if $(n +x)\theta \in I$, then $n \theta \in I-x$, whence, we can replace the interval $[N/M]$ by $[-N',N'] \cap \mathbb{Z}$ where $N' = N/(10M)$.

Now, let $I'$ be an interval centered at the same point as $I$ with $\mu_{\TT}(I')  = \mu_{\TT}(I)/100$ and let $I''$ be an interval centered at $0$ with $\mu_{\TT}(I'') = \mu_{\TT}(I)/100$. Next, let
    \[ B' =  \{ n \in [-N',N'] \cap \mathbb{Z}: n {\theta} \in I'  \}  \ \ \text{and} \ \ B'' =  \{ n \in [-N',N'] \cap \mathbb{Z}: n {\theta} \in I''  \}. \]
 Since $\mu_{\TT}(I) = \mu(\varphi^{-1}(I))$, we may use Lemma~\ref{lem: lower bound on phi^{-1}(I)} to get that $\mu_{\TT}(I) \geq \alpha/4$. Therefore, by replacing $\1_{B'}$ with a suitably Lipschitz function majorised by $\1_{B'}$ if necessary, we may use Lemma~\ref{distribution-integral-a} and the fact that $\theta$ is $(c\cF(M)/M,N)$-irrational to deduce that $B'$ is non-empty.  Next, we use Lemma \ref{geomnumber} to find a proper $d$-dimensional progression $P' \subseteq B''$ such that $d \leq 2$ and 
 \[ |P'| \gg \mu_{\TT}(I'') N' \gg  \mu_{\TT}(I'') N/M \gg \mu(\varphi^{-1}(I)) |I_{a,i}| \gg |B|, \]
 with the last inequality following from \eqref{coll}. Now, given some $b \in B'$, we claim that $b+ P' \subseteq B$. Indeed, for any $p \in P'$, we have
 \[ (b+p)\theta = b\theta + p \theta \in I' + I'' \subseteq I.\]
Setting $Q = b + P'$ concludes our proof. 
\end{proof}

Using \eqref{coll}, we see that $|B|\gg N\alpha/qM \gg_{\varepsilon, \cF} N \gg_{\varepsilon,\varepsilon', \delta } N$, which in turn gives us that
\[ |Q| \gg_{\varepsilon, \varepsilon', \delta} N. \]
Finally, note that since $|Q| \gg  |B|$ and since 
\[ |B \setminus A| = |B| - |B \cap A| \ll \varepsilon' |B|,\]
we may choose $\varepsilon'$ sufficiently small to deduce that 
\[ |Q \cap A| = |Q| - |Q \setminus A| \geq |Q| - |B \setminus A| \geq (1 - O(\varepsilon'))|Q|.  \]
Depending on whether $d$ is $1$ or $2$, the above implies the conclusion mentioned in part \eqref{it2} or \eqref{it3} of Theorem \ref{m1}. This concludes the proof of Theorem \ref{m1}.


\appendix

\section{The abelian arithmetic regularity lemma and related results} \label{arithmeticregularityappendix}

In this appendix we record the $U^2$ version of the arithmetic regularity lemma, together with some auxiliary lemmas that are used in the paper. Recalling the preliminary definitions given at the beginning of \S\ref{reglemmasteup}, the abelian arithmetic regularity lemma can be stated as follows, see \cite[Lemma A.2]{EGM2014} or \cite[Theorem 7]{Eb2015}.

\begin{lemma} \label{abelianarithmetic}
Let $f: [N] \to  [0,1]$ be a function, let $\varepsilon>0$ be a real number, and let $\mathcal F:(0, \infty) \to (0, \infty)$ be a growth function. Then there exists $M \ll_{\varepsilon, \cF} 1$ and a decomposition
\[
f=f_{\mathrm{struct}}+f_{\mathrm{sml}}+f_{\mathrm{unf}}
\]
of $f$ into functions $f_{\mathrm{struct}} : [N] \to [0,1]$ and $f_{\mathrm{sml}}, f_{\mathrm{unf}}: [N] \to [-1,1]$ such that $\n{f_{\mathrm{sml}}}_{\ell^2} < \varepsilon$ and $\n{f_{\mathrm{unf}}}_{U^2} < 1/\cF(M)$ and 
\[ f_{\mathrm{struct}}(n) = F(n/N, n \ {\rm mod} \ q, n \vec{\theta}) \]
for every $n \in [N]$, where $F: [0,1] \times (\mathbb{Z}/q\mathbb{Z}) \times \TT^d$ is an $M$-Lipschitz function, $q,d \leq M$ are positive integers and $\vec{\theta} \in \TT^d$ is $(\cF(M), N)$-irrational.
\end{lemma}

We refer the reader to \cite[Appendix A]{EGM2014} and \cite{Eb2015} for a nice introduction to the abelian arithmetic regularity Lemma. Eberhard--Green--Manners \cite{EGM2014} combined this with a nice discretisation of $[0,1]$ and some further analytic manoeuvres to obtain Lemma \ref{arithlem}. 

In our proofs, we will require the following properties of Lipschitz functions and functions with small $U^2$ norm. The first of these is \cite[Lemma A.3]{EGM2014}.

\begin{lemma} \label{distribution-integral-a}
Let $\delta, \eta>0$,  let $M>0$, let $d \in \mathbb{N}$ and let $A>0$ be sufficiently large in terms of $M, d, \eta, \delta$. Let $\vec{\theta} \in \TT^d$ be $(A, N)$-irrational, let $F : \TT^d \rightarrow \mathbb{C}$ be an $M$-Lipschitz function and let $P \subset [N]$ be an arithmetic progression of length at least $\eta N$. Then
\[ 
\left|\frac{1}{|P|}\sum_{n \in P} F(n\vec{\theta}) - \int_{\TT^d} F(\vec{\gamma})  d\vec{\gamma} \right| \leq \delta.
\]
\end{lemma}

The following is \cite[Lemma A.8]{EGM2014}.

\begin{lemma} \label{gowers-progressions}
Suppose that $f : [N]\rightarrow \mathbb{C}$ is a function, and $P \subset [N]$ is a progression of length at least $\eta N$. Then 
\[
\left|\frac{1}{|P|}\sum_{n \in P} f(n)\right| \ll \eta^{-1}\| f \|_{U^2}.
\]
\end{lemma}

Finally, we will also use \cite[Lemma 4.5]{EGM2014} and this is recorded below.

\begin{lemma} \label{proportionlem}
    Let $\eta>0$, let $\cF$ be a growth function growing sufficiently rapidly in terms of $\eta$. Then for any $M$-Lipschitz function $F: \mathbb{T}^d \to [0,1]$, any $(\cF(M), N)-$irrational $\vec{\theta} \in \mathbb{T}^d$, any arithmetic progression $P \subseteq [N]$ with $|P| \geq N/M^2$, one has
    \[ |P|^{-1} |\{n \in P : F(n\vec{\theta}) > \eta \}| \geq \mu(\vec{x} \in \mathbb{T}^d : F(\vec{x}) > 2\eta\} ) - \eta. \]
\end{lemma}

\bibliographystyle{amsbracket}
\providecommand{\bysame}{\leavevmode\hbox to3em{\hrulefill}\thinspace}

\end{document}